\pdfoutput=1

\documentclass[a4paper,11pt,titlepage=false]{scrartcl}

\usepackage{microtype}
\usepackage[utf8]{inputenc}
\usepackage[backend=bibtex,giveninits=true,urldate=long,maxbibnames=99]{biblatex}
\usepackage{mathtools}
\usepackage{amsthm}
\usepackage{amssymb}
\usepackage{tikz-cd}
\usepackage[colorlinks=true,linkcolor=blue]{hyperref}
\usepackage{cleveref}
\usepackage{comment}
\usepackage{stmaryrd}
\usepackage[new]{old-arrows}
\usepackage[automark]{scrlayer-scrpage}
\usepackage{csquotes}

\binoppenalty = \maxdimen
\relpenalty = \maxdimen

\DeclareFieldFormat{postnote}{#1}
\DeclareFieldFormat{multipostnote}{#1}

\DeclareFieldFormat[article,unpublished,incollection,thesis]{citetitle}{\mkbibitalic{#1}\isdot}
\DeclareFieldFormat[article,unpublished,incollection,thesis]{title}{\mkbibitalic{#1}\isdot}
\DeclareFieldFormat[article,book,collection,incollection]{number}{\mkbibbold{#1}}
\DeclareFieldFormat[article,collection,incollection]{volume}{\mkbibbold{#1}}
\DeclareFieldFormat[book]{pagetotal}{#1~\ppno}

\newtheorem{lemma}{Lemma}[section]
\newtheorem{proposition}[lemma]{Proposition}
\newtheorem{theorem}[lemma]{Theorem}

\newtheorem{corollary}[lemma]{Corollary}
\theoremstyle{definition}
\newtheorem{construction}[lemma]{Construction}
\newtheorem{definition}[lemma]{Definition}
\newtheorem{notation}[lemma]{Notation}
\newtheorem{remark}[lemma]{Remark}
\newtheorem{example}[lemma]{Example}

\DeclareMathOperator{\Ab}{\mathbf{Ab}}
\DeclareMathOperator{\Cell}{Ch}
\DeclareMathOperator{\coeq}{coeq}
\DeclareMathOperator{\cof}{cof}
\DeclareMathOperator*{\colim}{colim}

\DeclareMathOperator{\ev}{ev}
\DeclareMathOperator{\fib}{fib}
\DeclareMathOperator{\Fr}{Fr}
\DeclareMathOperator{\Fun}{Fun}

\DeclareMathOperator{\ho}{ho}
\DeclareMathOperator*{\holim}{holim}
\DeclareMathOperator{\Hom}{Hom}
\DeclareMathOperator{\LieGrpd}{\mathbf{LieGrpd}}
\DeclareMathOperator{\id}{id}
\DeclareMathOperator{\incl}{incl}
\DeclareMathOperator{\im}{im}
\DeclareMathOperator{\KT}{\mathbf{K}}
\DeclareMathOperator{\KU}{\mathbf{KU}}
\DeclareMathOperator{\map}{map}

\DeclareMathOperator{\Mfld}{\textbf{Mfld}}
\DeclareMathOperator{\nerve}{N}

\DeclareMathOperator{\Orb}{\mathbf{Orb}}
\DeclareMathOperator{\Orbfld}{\mathbf{Orbfld}}
\DeclareMathOperator{\op}{op}
\DeclareMathOperator{\Repl}{R}
\DeclareMathOperator{\RU}{\mathbf{RU}}
\DeclareMathOperator{\sh}{sh}

\DeclareMathOperator{\SL}{SL}
\DeclareMathOperator{\Sp}{\mathbf{Sp}}
\DeclareMathOperator{\spa}{span}
\DeclareMathOperator{\Spc}{\mathbf{Spc}}

\DeclareMathOperator{\Top}{\mathbf{T}}
\DeclareMathOperator{\TopGrpd}{\mathbf{TGrpd}}

\newcommand{\dash}{\textrm{-}}
\newcommand{\defset}[2]{\left\{#1 \mid #2 \right\}}
\newcommand{\B}{\mathbb{B}}
\newcommand{\Cent}{C}

\newcommand{\dlimwithout}{\lim\nolimits^1} 
\newcommand{\dlimwithinline}[2]{ \lim\nolimits^1_{#2} }
\newcommand{\dlimwithdisplay}[2]{ \mathop{\lim\nolimits^1}_{#2 \phantom{1}} } 
\newcommand{\dlimwith}[2]
{
	\mathchoice
	{\dlimwithdisplay_{#2}} 
	{\dlimwithinline_{#2}}  
	{\dlimwithinline_{#2}}  
	{\dlimwithinline_{#2}}  
}
\makeatletter
\newcommand{\dlim}{
	\@ifnextchar{_}
	{\dlimwith}
	{\dlimwithout}
}
\makeatother

\newcommand{\EC}{E}
\newcommand{\fin}{\mathcal{F}in}
\newcommand{\G}{\mathbb{G}}
\newcommand{\Leftrightarrows}{%
	\mathrel{%
		\offinterlineskip
		\vcenter{%
			\hbox{$\longrightarrow$}
			\vskip 0.1ex 
			\hbox{$\longleftarrow$}
		}
	}
}
\newcommand{\longrightrightarrows}{%
	\mathrel{%
		\offinterlineskip
		\vcenter{%
			\hbox{$\longrightarrow$}
			\vskip 0.1ex 
			\hbox{$\longrightarrow$}
		}
	}
}
\newcommand{\GH}{\mathcal{GH}}
\newcommand{\gmap}{\mathbb{M}\mathrm{ap}}
\newcommand{\gq}{\hspace{-2.5pt}\sslash\hspace{-1.5pt}}
\newcommand{\HB}{\mathbb{H}}

\newcommand{\LD}{\mathbb{L}}

\newcommand{\LI}{\mathcal L}
\newcommand{\effli}{{\mathcal{L}}_{\operatorname{eff}}}
\newcommand{\lieli}{{\mathcal{L}}_{\operatorname{Lie}}}
\newcommand{\LB}{\mathbf{L}}

\newcommand{\noloc}{%
	\nobreak
	\mskip6mu plus1mu
	\mathpunct{}%
	\nonscript
	\mkern-\thinmuskip 
	{:}%
	\mskip2mu
	\relax
}
\newcommand{\OC}{\mathbf{O}}
\newcommand{\OG}{O}
\newcommand{\Ogl}{{\mathbf{O}_{\operatorname{gl}}}}
\newcommand{\Or}[1]{\mathcal O_{#1}}
\newcommand{\Q}{\mathbb{Q}}
\newcommand{\pr}{\mathrm{pr}}
\newcommand{\R}{\mathbb R}
\newcommand{\RD}{\mathbb{R}}
\newcommand{\setdef}[2]{\left\{#1 \,\middle|\, #2\right\}}
\newcommand{\SH}{\mathcal{SH}}

\newcommand{\SO}{SO}
\newcommand{\SMC}{\mathcal M}
\newcommand{\T}{T}
\newcommand{\tg}{e}
\newcommand{\todo}[1]{\marginpar{#1}}
\newcommand{\UF}{U}
\newcommand{\Z}{\mathbb{Z}}

\newcommand{\defn}[1]{\emph{#1}}
\newcommand\res[2]{{
		\left.\kern-\nulldelimiterspace 
		#1 
		\vphantom{\big|} 
		\right|_{#2} 
}}

\newcommand{\LLO}{L}
\newcommand{\RLO}{R}
\bibliography{references.bib}

\setuptoc{toc}{leveldown}

\title{Orbifolds, Orbispaces and \linebreak Global Homotopy Theory}
\author{Branko Juran}

\begin{document}
\maketitle
\begin{abstract}
	Given an orbifold, we construct an orthogonal spectrum representing its stable global homotopy type. Orthogonal spectra now represent orbifold cohomology theories which automatically satisfy certain properties as additivity and the existence of Mayer-Vietoris sequences. Moreover, the value at a global quotient orbifold $M \gq G$ can be identified with the $G$-equivariant cohomology of the manifold $M$.
	Examples of orbifold cohomology theories which are represented by an orthogonal spectrum include Borel and Bredon cohomology theories and orbifold $K$-theory.
	This also implies that these cohomology groups are independent of the presentation of an orbifold as a global quotient orbifold. 
\end{abstract}
\tableofcontents
\section{Introduction}
\setcounter{secnumdepth}{2}
The purpose of this paper is to use global homotopy theory for studying the homotopy theory of orbifolds as suggested by Schwede in \cite[Preface]{sch18}.

Orbifolds are objects from geometric topology which model spaces that locally arise as quotients of smooth manifolds under smooth group actions. Special cases are \defn{global quotient orbifolds} $M \gq G$ of a smooth manifold $M$ under a smooth action of a compact Lie group~$G$. 

Orbifolds were first defined by Satake \cite{sat56} who called them $V$-manifolds. 
The definition is geometric and generalizes the definition of a manifold by endowing the euclidean charts with a finite group action. This notion was later used and popularized by Thurston \cite{thu79}, who started calling them orbifolds and used them to study three-manifolds. 
In a more general approach, one defines orbifolds as special kind of Lie groupoids. 
Connections between groupoids and orbifold were first studied by Haefliger \cite{hae82}. The definition which is used nowadays is due to Moerdijk and Pronk \cite{mp97}. 

It is essential that an orbifold carries more information than just the underlying topological quotient space: It also keeps track of the group actions involved. The homotopy type of a global quotient orbifold~$M \gq G$ should relate to the $G$-equivariant homotopy type of the underlying manifold $M$. For example, the cohomology of $M \gq G$ is usually defined as the $G$-equivariant cohomology of $M$. However, the group $G$ may vary for different orbifolds and two quotients can represent the same orbifold, e.g for $H$ a subgroup of $G$, the quotients $M \gq H$ and $(G \times_H M) \gq G$ are equivalent as orbifolds.
The presentation of an orbifold $X$ as a global quotient orbifold $M \gq G$ hence is non-instrinsic data of the orbifold $X$. However, an orbifold cohomology theory should capture the $G$-equivariant behaviour of the orbifold $X$. 
A conceptual treatment of \defn{orbifold cohomology theories} must necessarily pay attention to the interplay between different equivariant cohomology theories. This leads us to global homotopy theory. 

Global homotopy theory studies homotopy types with simultaneous, compatible group actions of all compact Lie groups. A particular framework for stable global homotopy theory was developed by Schwede in his book \cite{sch18}. He uses the category of orthogonal spectra together with \defn{global equivalences}. The notion of global equivalences is much more restrictive than the usual one of stable equivalences between orthogonal spectra. Two orthogonal spectra are globally equivalent if they represent the same \defn{global homotopy type} and not just the same non-equivariant homotopy type. Localizing the category of orthogonal spectra at the global equivalences yields the \defn{stable global homotopy category}~$\GH$. There is a functor~$\GH \to G\dash\SH$ from the stable global homotopy category to the $G$-equivariant stable homotopy category for any Lie group $G$ which assigns the \defn{underlying $G$-equivariant homotopy type} to a global homotopy type. 

Similarly, the category of orthogonal spaces together with global equivalences of orthogonal spaces models unstable global homotopy types. Orthogonal spectra, i.e. stable global homotopy types, represent cohomology theories on these orthogonal spaces.
We define a functor from orbifolds to orthogonal spaces.
Using this functor, an orthogonal spectrum does also represent a cohomology theory on orbifolds.

This functor sends a global quotient orbifold $M \gq G$ to the so-called semifree orthogonal space $\LB_{G,V}M$ which satisfies the desired property: For $E$ any orthogonal spectrum, the cohomology group $E^k(\LB_{G,V}M)$ actually bijects with $E_G^k(M)$, the $G$-equivariant cohomology of $M$ represented by the underlying $G$-equivariant homotopy type $E_G$ of~$E$. Moreover, our construction sends equivalences of orbifolds to global equivalences of orthogonal spaces. It immediately follows that orbifold cohomology groups are independent of the presentation of the orbifold as a global quotient orbifold. We also verify other properties of these cohomology theories including  additivity and the existence of Mayer-Vietoris sequences and Atiyah-Hirzebruch spectral sequences. 

Examples of equivariant cohomology theories which are represented by orthogonal spectra in the global sense include equivariant $K$-theory, Borel cohomology theories and Bredon cohomology theories. More precisely, the underlying $G$-equivariant homotopy types of these spectra represent the respective $G$-equivariant cohomology for all compact Lie groups $G$. These cohomology theories hence are examples of \defn{orbifold cohomology theories} in our sense. Particularly, we show that Bredon cohomology extends to a well-behaved cohomology theory on all orbifolds. These examples are discussed in greater detail in \Cref{htorb}.
\addtocontents{toc}{\protect\setcounter{tocdepth}{1}}
\subsection*{Organization of the paper}
In \Cref{orbifolds} we recall two definitions of orbifolds, \defn{effective orbifolds} as defined by Satake and the more general concept of an orbifold as a Lie groupoid. In \Cref{orbispaces}, we recall two definitions of model categories of orbispaces: the category $\LI\dash\Top$ as defined by Schwede \cite{sch19} and the category $\Top_{\Orb}$ which is due to Gepner and Henriques \cite{gh07}. Moreover, we recall a result by Körschgen \cite{kor16} establishing a zig-zag of Quillen equivalences between these two model categories. We establish two constructions for associating an orbispace to an orbifold: In \Cref{effli}, we use effective orbifolds and the the category $\LI\dash\Top$ for a non-functorial but easily computable construction. In \Cref{grouporborb}, we define a functor from orbifolds to the category $\Top_{\Orb}$. We will use the zig-zag of Quillen equivalences to extend this functor to a functor to the category $\LI\dash\Top$ in \Cref{comparison2}. Finally, we associate an orthogonal space to an orbifold and define what an \defn{orbifold cohomology theory} is in \Cref{SGHB}. 
\subsection*{Acknowledgment} This paper arose from the author’s Bachelor's thesis supervised by Stefan Schwede at the University of Bonn. I would like to thank Stefan Schwede for the suggestion of the topic and his advice and support while and after writing it. I am grateful to Benjamin Böhme for co-correcting my thesis. Moreover, I wish to thank Thorben Kastenholz for various helpful discussions on the topics of this paper. I am indebted to Urs Flock, Jan Jendrysiak, Christian Kremer, Sebastian Meyer and Jendrik Stelzner for proofreading an earlier version of this paper. 
\subsection*{Conventions}
Throughout this paper, $\Top$ denotes the category of compactly generated, weak Hausdorff spaces together with continuous maps, as defined in \cite{str09}. This definition was first used by McCord in his paper \cite{mcc69}. Other sources for details on this topic are \cite[Sec. 7.9]{tom08}, \cite[App.~A]{lew78} and \cite[App.~A]{sch18}. A \emph{space} is always assumed to be an object of $\Top$. 
Moreover, $\Top_\ast$ denotes the category of based, compactly generated, weak Hausdorff spaces together with based continuous maps.

Let $G$ be a topological group, i.e. a group object in $\Top$.
We write~$G\dash\Top$ for the category of (left) $G$-spaces together with $G$-equivariant continuous maps. Moreover,~$G\dash\Top_\ast$ denotes the category of based $G$-spaces together with based~$G$-equivariant continuous maps.

A \defn{manifold} is always assumed to be a smooth manifold which is second-countable and Hausdorff. The category of smooth manifold together with smooth maps is denoted by $\Mfld$. We require all occurring group actions of Lie groups on manifolds to be smooth.
\addtocontents{toc}{\protect\setcounter{tocdepth}{2}}
\section{Orbifolds} \label{orbifolds}
Firstly, we recall the geometric definition of an effective orbifold in \Cref{efforb}. Afterwards, we explain the definition of an orbifold via Lie groupoids in \Cref{grouporb}.

In \Cref{grpdtoeff} we explain how orbifolds can be seen as a generalization of effective orbifolds. Nevertheless, we want to warn the reader that we usually do not consider \enquote{effective} as a property of an orbifold; \enquote{effective orbifolds} should be seen as a different kind of objects than \enquote{orbifolds}.
\subsection{Effective Orbifolds} \label{efforb}
The notion of an effective orbifold historically was the very first definition of orbifolds and is due to Satake \cite{sat56}.
In the following section we discuss some classical constructions and results. 

Recall that we require manifolds to be smooth manifolds and Lie group actions on manifolds to be smooth.
The following definitions are taken from \cite[Sec. 1.1]{alr07} where the interested reader may also find more details on this topic. 

We need to restrict the class of allowed group actions: The action of the group should have a geometric impact on the quotient space, e.g. we don't want to allow a trivial action of a non-trivial group.
\begin{definition}[Effective]
	A group action of a group $G$ on a space $A$ is called \defn{effective} or \defn{faithful} if $g.x=x$ for all points $x$ in $A$ implies that $g$ is the neutral element of $G$. 
\end{definition}
\begin{definition}[Orbifold chart, orbifold atlas, refinement, equivalent charts] \label{definitionefforbifold} Let $A$ be a space and let $n$ be a non-negative integer. 
	
	An \defn{$n$-dimensional orbifold chart} on $A$ is a triple $(U,G,\varphi)$ where $G$ is a finite group, $U$ is a connected open subset of $\R^n$ together with a smooth, effective $G$-action on $U$ and~$\varphi \colon U \to A$ is a $G$-invariant continuous map such that the induced map $U/G \to A$ is a homeomorphism onto an open subset of $A$.
	
	An \defn{embedding} $\lambda \colon (U_1,G_1,\varphi_1) \hookrightarrow (U_2,G_2,\varphi_2)$ between two $n$-dimensional orbifold charts is a smooth embedding $\lambda \colon U_1 \to U_2$ such that $\varphi_2 \lambda = \varphi_1$.
	
	An $n$-dimensional \defn{orbifold atlas} on a space $A$ is a set $\mathcal U$ of $n$-dimensional orbifold charts on $A$ covering $A$ such that the following holds: For every pair of orbifold charts~$(U_1,G_1,\varphi_1)$ and $(U_2,G_2,\varphi_2)$ in $\mathcal U$ and every point $x$ in $\varphi_1(U_1) \cap \varphi_2(U_2)$, there is a third orbifold chart $(U_3,G_3,\varphi_3)$ in $\mathcal U$ such that $\varphi_3(U_3)$ contains $x$ and there are em\-bed\-dings~$(U_3,G_3,\varphi_3) \hookrightarrow (U_1,G_1,\varphi_1)$ and~$(U_3,G_3,\varphi_3) \hookrightarrow (U_2,G_2,\varphi_2)$.
	
	An orbifold atlas $\mathcal U$ on $A$ is a \defn{refinement} of another orbifold atlas $\mathcal V$ on $A$ if for every orbifold chart $(U,G,\varphi)$ in  $\mathcal U$ there is an orbifold chart in $(V,H,\psi)$ in $\mathcal V$ and an embedding~$(U, G,\varphi) \hookrightarrow (V,H,\psi)$.
	
	Two orbifold charts are said to be \defn{equivalent} if they admit a common refinement.
\end{definition}
\begin{definition}[Effective orbifold]
	An \defn{effective orbifold} of dimension $n$ is a tuple~$\chi=(A,[\mathcal U])$ where $A$ is a second-countable, Hausdorff space and $[\mathcal U]$ is an equivalence class of $n$-dimensional orbifold atlases on $A$.
	We write $|\chi|=A$ for the \defn{underlying space} of $\chi$.
\end{definition}
We don't need to require that the underlying space of an effective orbifold is compactly generated and weak Hausdorff. This is automatically true because the underlying space is locally compact. We moreover assumed the underlying space to be second-countable and Hausdorff. This implies that the underlying space is also paracompact.

Every orbifold atlas is contained in a unique maximal orbifold atlas. As for manifolds, we tend to work with maximal orbifold atlases in proofs.
\begin{definition}[Effective orbifold map, diffeomorphism] \label{diffeoeff}
	Let $\chi_1$ and $\chi_2$ be two effective orbifolds with maximal orbifold atlases $\mathcal U_1$ and $\mathcal U_2$.
	A \defn{map} between $\chi_1$ and $\chi_2$ is a continuous map $f \colon |\chi_1| \to |\chi_2|$ between their underlying spaces such that for all points~$x$ in~$|\chi_1|$, there is an orbifold chart $(U_1,G_1,\varphi_1)$ in $\mathcal U_1$ around $x$ and an orbifold chart~$(U_2,G_2,\varphi_2)$ in $\mathcal U_2$ around $f(x)$  such that $f$ lifts to a smooth map $\tilde f \colon U_1 \to U_2$, that is~$\varphi_2 \tilde f = f \varphi_1$.
	
	A \defn{diffeomorphism} of effective orbifolds is a map of effective orbifolds admitting an inverse map of effective orbifolds. Two effective orbifolds are said to be \defn{diffeomorphic} if there is a diffeomorphism between them.
\end{definition}
\begin{example} \label{exampleorbifolds}
	We briefly discuss two examples of effective orbifolds. 
	\begin{itemize}
		\item Every manifold can be seen as an effective orbifold. The charts are the same, but endowed with the trivial group action.
		\item More generally, let $G$ be a finite group which acts effectively on a connected manifold $M$. Then $M/G$ admits an orbifold atlas: For every point $x$ in $M$ we can choose a $G$-invariant euclidean neighborhood of $x$. The pullback of the $G$-action to the euclidean space gives rise to an orbifold chart around $[x] \in M/G$. The resulting effective orbifold is denoted by $M \gq G$.
	\end{itemize}
\end{example}
We want to generalize this construction to quotients of manifolds under compact Lie group actions. We therefore need to study the local behavior of these actions. 
\begin{definition}[Slice, linear slice] \label{defslice}
	Let $G$ be a compact Lie group acting on a manifold $M$ and let $x$ be a point in $M$. Let $G_x$ denote the stabilizer group of $x$. A $G_x$-invariant subset $S$ of $M$ containing $x$ is called a \defn{slice} at~$x$ if the map \begin{align*}
	G \times_{G_x} S & \longrightarrow M \,, \\
	[g,s] & \longmapsto g.s
	\end{align*} is a diffeomorphism onto an open neighborhood of $G.x$ in $M$. 
	
	We say that this slice is \emph{linear} if the $G_x$-space $S$ is $G_x$-equivariantly homeomorphic to an orthogonal $G_x$-space, that is $S \cong \R^n$ via a $G_x$-equivariant homeomorphism where the~$G_x$-action on $\R^n$ is orthogonal.
\end{definition} 
For $M$ a manifold together with an action of a compact Lie group $G$, there exists a linear slice at any point $x$ in $M$, see \cite[Thm. VI.2.4]{bre72}. This implies that we can assume all orbifold charts to be linear, that is an orbifold chart $(U,G,\varphi)$ where $G$ acts orthogonally on $U = \R^n$.

Before defining an effective orbifold structure on a quotient space, we need to restrict the class of allowed compact Lie group actions.
\begin{definition}[Almost free, effective on slices]
	Let $G$ be a group acting on a space~$A$. We say that $G$ acts \defn{almost freely} if the stabilizer group $G_x$ is finite for all $x$ in $A$. 
	
	Moreover, an action of a compact Lie group on a manifold is \defn{effective on slices} if the action of $G_x$ on any linear slice $S$ at any point $x$ in $A$ is effective.
\end{definition}
Group actions which are effectively on slices are in particular effective. Moreover, all linear slices at a point $x$ are $G_x$-equivariantly homeomorphic. Therefore, it is enough to check that the action of $G_x$ on one particular slice is effective.
\begin{construction}[Effective global quotient orbifold] \label{globalquotient}
	Let $M$ be a manifold and let $G$ be a compact Lie group which acts almost freely and effectively on slices on $M$. We define an effective orbifold structure on $M/G$ as follows:
	
	Let $x$ be a point in $M$ and let $S$ be a linear slice at $x$. There is a~$G_x$-equivariant diffeomorphism $\R^n \cong S$ such that $G_x$ acts orthogonally on $\R^n$. Then $(\R^n,G_x,\varphi)$ is an orbifold chart on $M/G$ where~$\varphi \colon \R^n \to M/G$ is induced by the inclusion of $S \cong \R^n$ into $M$. It is proven in \cite[Prop. II.4.7]{bre72} that the induced map $\R^n/G_x \cong S/G_x \to M/G$ is actually a homeomorphism onto the open $G$-invariant subset $G(S)$.
	
	Let $\mathcal U$ be the the set of all these orbifold charts. We call $(M/G,[\mathcal U])$ an \defn{effective global quotient orbifold} and denote it by $M \gq G$.
\end{construction}
We omit the verification that $\mathcal U$ actually is an orbifold atlas on $M/G$. The author is not aware of an explicit reference with a proof of this well-known fact, but one can consult \cite[Sec. 2.4]{mm03} for a slightly different construction using foliation theory.
\begin{remark}
	Although it is sometimes stated differently in the literature, we need to require that the group acts effectively on slices. It is not enough to ask for an effective group action. Despite trivial counterexamples where the group action is trivial on one connected component, there are also more interesting examples: Let $S_3$ be the third permutation group and let $H$ be an order $2$ subgroup of $S_3$. Then the action of~$S_3$ on~$S_3/H$ is almost free and effective. The orbit space of this action is just a single point, but all stabilizers have order $2$. 
	
	One can also construct an example for a connected Lie group by replacing $S_3$ by~$\SO(3)$.
\end{remark}
In fact, every effective orbifold is diffeomorphic to an effective global quotient orbifold, see \Cref{efforbaregq}. We recall an explicit construction for this by defining the manifold of orthonormal frames of an effective orbifold in \Cref{constrframebundle}. 
\begin{definition}[Riemannian metric]
	Let $\chi$ be an effective orbifold with maximal orbifold atlas $\mathcal U = (U_i,G_i,\varphi_i)_{i \in I}$. A \defn{Riemannian metric} on $\chi$ is a collection $(\rho_i)_{i \in I}$ of Riemannian metrics $\rho_i$ on $U_i$ such that every embedding of orbifold charts becomes an isometry.  
\end{definition}
Every effective orbifold admits a Riemannian metric, see \cite[Prop. 2.20]{mm03}. It follows from the $G_i$-invariance of the map $\varphi_i \colon U_i \to |\chi|$ that the Riemannian metric $\rho_i$ on $U_i$ is $G_i$-invariant for all $i \in I$.
\begin{construction}[Manifold of orthonormal frames] \label{constrframebundle}
	Let $\chi$ be an effective orbifold of dimension $n$ with maximal orbifold atlas $(U_i,G_i,\varphi_i)_{i \in I}$ together with a Riemannian metric on $\chi$.
	We will recall the essential ideas of the construction of the manifold of orthonormal frames from \cite[Sec. 2.4]{mm03}. We just slightly change the definition to obtain a left $\OG(n)$-space, i.e. an orthonormal frame is defined as a linear isometry $\T_xM \to \R^n$ instead of a linear isometry $\R^n \to \T_xM$. 
	
	The manifold of orthonormal frames $\Fr(U_i)$ admits an $(\OG(n) \times G_i)$-action via  \[(A,g).(x,e)=\left(g.x,A \circ e \circ \left(\T_xg \right)^{-1}\right) \] for $g \in G_i$ and $A \in \OG(n)$. 
	The restricted $G_i$-action on $\Fr(U_i)$ is free. The quotient~$\Fr(U_i)/G_i$ hence canonically admits the structure of a smooth $\OG(n)$-manifold.
	
	
	Every embedding of orbifold charts $\lambda \colon (U_i ,G_i,\varphi_i) \hookrightarrow (U_j,G_j,\varphi_j)$ gives an isometry~$\lambda \colon U_i \hookrightarrow U_j$ which induces a smooth embedding $\hat \lambda \colon \Fr(U_i) \hookrightarrow \Fr(U_j)$ by \[\hat \lambda (x,e) = \left(\lambda(x),e \circ \left(\T_x \lambda\right)^{-1}\right) \] for a frame $(x,e)$ at a point $x$ in $U_i$. 
	The map $\hat \lambda$ induces an $\OG(n)$-equivariant smooth embedding \[\lambda_\ast \colon \Fr(U_i)/G_i \longrightarrow \Fr(U_j)/G_j \,. \] 
	Gluing together all quotients $Fr(U_i)/G_i$ along these maps yields the \defn{manifold of orthonormal frames} $\Fr(\chi)$ of $\chi$. 
\end{construction}
\begin{example}[Manifold of orthonormal frames of $M \gq G$] \label{framebundleofMG}
	Let $M$ be an $m$-dimensional manifold together with an almost free action of a compact $n$-dimensional Lie group $G$ on $M$ which is effective on slices.  
	It is claimed in \cite[Prop. 1.25]{alr07} that~$\Fr(M \gq G)$ is diffeomorphic to $\Fr(M)/G$.
	This does not seem to be fully correct if $G$ is not discrete because the dimensions would not agree. The orbifold structure on $M/G$ just uses slices, that are transversals to the orbits of the group action. We must therefore also replace $\Fr(M)$ by the manifold $\Fr_G(M)$ of orthonormal $(m-n)$-frames which are orthogonal to the orbits to obtain an $\OG(n)$-equivariant diffeomorphism $\Fr(M \gq G) \cong \Fr_G(M)/G$, i.e. there is a commutative diagram: \[\begin{tikzcd}[sep=large]
	\Fr_G(M) \ar[r,"/\OG(m-n)"] \ar[d,"/G"] & M \ar[d,"/G"] \\
	\Fr(M \gq G)  \ar[r,"/\OG(m-n)"] & M/G
	\end{tikzcd} \]
\end{example}
The homotopy type of an effective orbifold $\chi$ often is defined as the $\OG(n)$-equivariant homotopy type of $\Fr(\chi)$ and we will follow this idea in \Cref{effli}. One should hence verify that $\Fr(\chi)$ does not depend on the choice of the Riemannian metric up to $\OG(n)$-equivariant diffeomorphism. The author is not aware of a proof in the literature. Therefore, we explain why the construction for manifolds described in \cite{bg92} behaves well under an additional group action and hence extends to effective orbifolds. 
Let $V$ be a finite-dimensional vector space together with two inner products $\rho$ and~$\sigma$ on $V$. Let $\hat \rho \colon V \to V^\ast$ and $\hat \sigma \colon V \to V^\ast$ denote the isomorphisms induced by $\rho$ and $\sigma$, respectively. We write \[b_\sigma^\rho = (\hat \rho)^{-1} \circ \hat \sigma \colon V \longrightarrow V \,.\]
Recall from linear algebra that a symmetric positive definite matrix $A$ has a unique symmetric positive definite square root which will be denoted by $A^{1/2}$. 
This square root is already uniquely determined by the fact that all eigenvalues are positive as can easily be seen by diagonalizing $A$.

We omit the proof of the following linear algebra fact. 
\begin{lemma} \label{forvs}
	Let $V$ be a finite-dimensional vector space together with two inner products $\rho$ and~$\sigma$ on $V$. 
	The linear map $b_\sigma^\rho \colon V \to V$ is self-adjoint with respect to $\rho$ and has positive eigenvalues only. Its unique square root with the same properties $(b_{\sigma}^\rho)^{1/2}$ is an isometry from $(V,\sigma)$ to $(V,\rho)$.
\end{lemma}
\begin{proposition} \label{formanifolds}
Let $\rho$ and $\sigma$ be two $G$-invariant Riemannian metrics on an open and connected subset $U$ of $\R^n$ together with an effective action of a finite group $G$ on $U$. For a point $x$ in $M$, we write $b_x$ for $b_{\sigma_x}^{\rho_x}$. The map \begin{align*}
B^\rho_\sigma \colon \Fr^\rho(U) & \longrightarrow \Fr^\sigma(U) \\
(x,e) & \longmapsto (x,e \circ b_x^{1/2})
\end{align*}
for $(x,e)$ an orthonormal frame at $x$ with respect to $\rho$ is an $(\OG(n) \times G)$-equivariant diffeomorphism.
\end{proposition}
\begin{proof}
The map $B^\rho_\sigma$ is well-defined by \Cref{forvs}. Moreover, the map obviously is $\OG(n)$-equivariant. We verify that it is also $G$-equivariant. 
First note that the invariance of the metric $\rho$ under the action map of an element $g$ of $G$ at a point $x$ in $U$ can be written as \[ \T^\ast_xg \circ \hat \rho_{g.x} \circ \T_xg = \hat \rho_x \] where $\T^\ast_{x}g \colon \T_{g.x}^\ast U \to \T_{x}^\ast U$ is induced by $\T_xg \colon \T_xU \to \T_{g.x}U$. An analogous statement holds for $\sigma$.
We conclude that \begin{align*}\left(\T_{x}g\right)^{-1} \circ b_{g.x} \circ \T_xg ={}& b_x \shortintertext{and hence} \left(\T_{x}g\right)^{-1} \circ b_{g.x}^{1/2} \circ \T_xg ={}& b_x^{1/2} \end{align*} because both sides are square roots of $b_x$ with positive eigenvalues only.
We can finally verify that \begin{align*} g.B_\sigma^\rho(x,e) ={}& \left(g.x,e \circ b_x^{1/2} \circ \left(\T_{x}g\right)^{-1}\right) \\ ={}& \left(g.x,e \circ \left(T_{x}g\right)^{-1} \circ b_{g.x}^{1/2}\right) \\ ={}& B_\sigma^\rho(g.(x,e))  \end{align*} holds true for any frame $(x,e)$ in $\Fr_x(U)$.

One can verify that the map is smooth by choosing local orthonormal frames and using the fact that the unique symmetric positive definite square root depends smoothly on the symmetric positive definite matrix. 
\end{proof}
\begin{theorem} \label{choiceofmetric}
Let $\rho$ and $\sigma$ be two Riemannian metrics on an effective orbifold $\chi$. Let $\Fr^\rho(\chi)$ and $\Fr^\sigma(\chi)$ denote the manifold of orthonormal frames with respect to the metric $\rho$ and $\sigma$, respectively. Then there exists an~$\OG(n)$-equivariant diffeomorphism between $\Fr^\rho(\chi)$ and $\Fr^\sigma(\chi)$.
\end{theorem}
\begin{proof}
Let $(U,G,\varphi)$ be an orbifold chart of $\chi$. The diffeomorphism \[B_\sigma^\rho \colon \Fr^\rho(U) \longrightarrow \Fr^\sigma(U) \] is $(\OG(n) \times G)$-equivariant by \Cref{formanifolds} and hence induces an $\OG(n)$-equivariant diffeomorphism \[\Fr^\rho(U)/G \longrightarrow \Fr^\sigma(U)/G \,. \] Moreover, the diffeomorphism $B_\sigma^\rho$ just depends on the inner product on each tangent space of $U$. These maps are hence compatible with the embeddings used in \Cref{constrframebundle} of the manifold of orthonormal frames. 
We obtain a map \[\Fr^\rho(\chi) \longrightarrow \Fr^\sigma(\chi) \] which is an $\OG(n)$-equivariant diffeomorphism.
\end{proof}
The diffeomorphism $\Fr^\sigma(\chi) \to \Fr^\rho(\chi)$ does not depend on any further choices. However, it should not be seen as a canonical choice for such a diffeomorphism: It does not behave well under composition, i.e. in general $ B^\sigma_\tau \circ B_\sigma^\rho$ is not equal to $B^\rho_\tau$ for three Riemannian metrics $\rho,\sigma$ and $\tau$ on $\chi$ because $A^{1/2}B^{1/2}$ is usually not equal to $(AB)^{1/2}$ for two matrices $A$ and $B$. 
\begin{corollary} \label{framebundleindependent}
Let $\chi_1$ and $\chi_2$ be two diffeomorphic effective orbifolds. There is an $\OG(n)$-equivariant diffeomorphism between the manifolds $\Fr(\chi_1)$ and $\Fr(\chi_2)$.
\end{corollary}
\begin{proof}
The statement is independent of the choices of the Riemannian metrics by \Cref{choiceofmetric}. Choose a Riemannian metric on $\chi_2$. We can pull back this metric along a diffeomorphism $f \colon \chi_1 \to \chi_2$: Let $(U,G,\varphi)$ be an orbifold chart of $\chi_1$. Then $(U,G,f \varphi)$ is an orbifold chart for $\chi_2$ and we may choose the Riemannian metric on $U$ for $(U,G,\varphi)$ to be the one on $U$ for $(U,G,f\varphi)$. The map $f$ locally lifts to the identity map on $U$ and hence induces a $\OG(n)$-equivariant diffeomorphism $\Fr(\chi_1) \to \Fr(\chi_2)$. 
\end{proof}
\begin{theorem} \label{efforbaregq}
Let $\chi$ be an effective orbifold. Then the action of $\OG(n)$ on $\Fr(\chi)$ is almost free and effective on slices. Moreover, there is a diffeomorphism \[\chi \cong \Fr(\chi)\gq\OG(n) \] of effective orbifolds.
\end{theorem}
\begin{proof}
Let $x$ be a point in $|\chi|$ and $(U,G,\varphi)$ a linear orbifold chart around $x$. We may in particular assume that $G_x=G$. There is a neighborhood of $x$ in $\Fr(\chi)$ which is diffeomorphic to $\Fr(U)/G$. By \Cref{formanifolds}, we may assume that $\Fr(U)$ is constructed by using the euclidean metric on $U$. The manifold $\Fr(U)$ is then equivariantly diffeomorphic to $U \times \OG(n)$ where $G$ acts via the action map on $U$ and via the differential on $\OG(n)$. This action map is given by multiplication with an element of $\OG(n)$. Therefore, the differential has constant value, an element of $\OG(n)$.
It follows that there is a $\OG(n)$-equivariant diffeomorphism \[\Fr(U)/G \cong \OG(n) \times_{G} U \,.\]
For any $A$ in $\OG(n)$, we obtain a linear slice $A.{U}$ at~$[A,x]$ in $\OG(n) \times_{G} U$ under the action of $\OG(n)$. This action in particular is almost free and effective on slices. Using the above diffeomorphism, we see that the action of $\OG(n)$ on $\Fr(\chi)$ is also almost free and effective on slices. We can also conclude that $(U,G,\varphi^\prime)$ is a linear orbifold chart of $\Fr(\chi)\gq \OG(n)$ where $\varphi^\prime$ is the composition \[U \longhookrightarrow \OG(n) \times_G U \cong \Fr(U)/G \longhookrightarrow \Fr(\chi) \longrightarrow \Fr(\chi)/\OG(n) \] where the first map sends $x \in U$ to $[\id,x]$.
We now obtain a homeomorphism \[\left(\Fr(U)/G \right)/\OG(n) \cong \left(\Fr(U)/\OG(n) \right)/G \cong U/G \] which lifts to a map of charts $\id \colon U \to U$ from $(U,G,\varphi^{\prime})$ to $(U,G,\varphi)$. These homeomorphisms are compatible with embeddings of orbifold charts and induce the desired diffeomorphism $\Fr(\chi) \gq \OG(n) \to \chi$.
\end{proof}
\subsection{Orbifolds as Lie Groupoids} \label{grouporb}
There are two disadvantages when working with effective orbifolds.
\begin{itemize}
	\item They only model quotients of effective group actions on manifolds. But there are important examples of ineffective quotients, e.g. the quotient $\mathbb{H}/\SL_2(\mathbb{Z})$ of the complex upper half plane $\mathbb{H}$ by the action of $\SL_2(\mathbb{Z})$, which is the moduli space of elliptic curves. The non-trivial element $-\id$ acts trivially. 
	\item There is no good definition of morphisms between two effective orbifolds. It is in particular difficult to define pullbacks of vector bundles: They are not defined for all morphisms. More details on pullbacks of vector bundles over effective orbifolds can be found in~\cite[Sec. 4.4]{cr02}.
\end{itemize}
Both problems can be fixed by defining orbifolds as a certain kind of Lie groupoids. An excellent overview of this approach can be found in Moerdijk's paper \cite{moe02}. A closely related approach, which won't be discussed here but which may be of interest to the reader as well, is the one using differentiable stacks. An introduction to this approach, which also recovers the definition of orbifolds as Lie groupoids, can be found in Lerman's paper \cite{ler08}.  
\begin{definition}[Topological groupoid] \label{deftopgrpd}
	A \defn{topological groupoid} is a groupoid object in~$\Top$, that is a tuple $\G=(\G_0,\G_1,s,t,e,i,\circ)$ where $\G_0$ and $\G_1$ are spaces and~$s \colon \G_1 \to \G_0$, $t \colon \G_1 \to \G_0$, $e \colon \G_0 \to \G_1$, $i \colon \G_1 \to \G_1$ and $\circ \colon \G_1 \times_{s,t} \G_1 \to \G_1$ are continuous maps such that the underlying sets of $\G_0$ (as objects) and $\G_1$ (as arrows) together with $s$ as the source map, $t$ as the target map, $e$ as the identity map, $i$ as the inverse map and $\circ$ as the composition map form a (small) groupoid.  
	
	A \defn{morphism} of topological groupoids $f \colon \G \to \HB$ is a pair of continuous maps $f_0 \colon \G_0 \to \HB_0$ and $f_1 \colon \G_1 \to \HB_1$ which is a functor between the underlying groupoids.
	A \defn{$2$-morphism} between two morphisms of topological groupoids $f,g \colon \G \to \HB$ is a continuous map $\G_0 \to \HB_1$ which is a natural transformation between $f$ and $g$. 
	
	The category of topological groupoids is enriched over itself. We topologize the set~$\gmap(\G,\HB)_0$ of morphisms between $\G$ and $\HB$ as a subspace of $\map(\G_1,\HB_1)$. Moreover,~$\gmap(\G,\HB)_1$ is the space of $2$-morphisms between~$\G$ and~$\HB$ topologized as a subspace of $\map(\G_1,\HB_1) \times \map(\G_1,\HB_1)\times \map(\G_0,\HB_1)$ where the inclusion into the first factor is given by the source functor, the one into the second factor by the target functor and the one into the third factor by the underlying map of the natural transformation itself.  
	
	The category of topological groupoids is denoted by $\TopGrpd$. 
\end{definition}
\begin{definition}[Lie groupoid]
	Similarly, a \defn{Lie groupoid} is tuple $\G=(\G_0,\G_1,s,t,e,i,\circ)$ where $\G_0$ and $\G_1$ are manifolds, $s \colon \G_1 \to \G_0$, $t \colon \G_1 \to \G_0$ are smooth submersions and~$e \colon \G_0 \to \G_1$, $i \colon \G_1 \to \G_1$ and $\circ \colon \G_1 \times_{s,t} \G_1 \to \G_1$ are smooth maps such that the underlying sets form a groupoid as above. 
	
	Morphisms and $2$-morphisms are defined just as for topological groupoids with smooth maps instead of continuous ones.
	The category of Lie groupoids is denoted by $\LieGrpd$.
\end{definition}
If $s$ is a submersion, then $t$ automatically is a submersion because $t=s \circ i$ and $i$ is a self-inverse diffeomorphism.
The requirement that $s,t \colon \G_1 \to \G_0$ are submersions ensures that the fibered product $\G_1 \times_{s,t} \G_1$ exists in $\Mfld$ and that the underlying space can be computed in $\Top$. 

The forgetful functor $\Mfld \to \Top$ induces a forgetful functor $\LieGrpd \to \TopGrpd$. In particular, every Lie groupoid can be regarded as a topological groupoid. We will not mention this when using this fact.  
\begin{example} \label{examplegroupoids}We will give some examples without explicitly mentioning or defining all maps.  
	\begin{itemize} 
		\item Let $G$ be a topological group (or Lie group, respectively). Let $\B G$ denote the topological groupoid (or Lie groupoid, respectively) with objects space $\B G_0=\ast$ and arrow space $\B G_1=G$. The identity map of $\B G$ sends $\ast$ to the identity of $G$. Composition and inverse map are given by the corresponding maps of $G$.
		\item More generally, let $G$ be a topological group and $A$ be a $G$-space. Define the \defn{action groupoid} $G \ltimes A$ with $(G \ltimes A)_0=A$ and $(G \ltimes A)_1=G \times A$. The map $s \colon G \times A \to A$ is the projection map and the map $t \colon G \times A \to A$ is the action map of $A$. 
		
		An arrow $g \colon x \to y$ in $G \ltimes A$ from $x$ to $y$ in $A$ may be described as an element~$g$ of $G$ such that $g.x=y$. The composition is defined by using the composition of~$G$.  
		
		The topological groupoid $G \ltimes A$ carries the structure of a Lie groupoid if $G$ is a Lie group acting smoothly on a manifold $M=A$.
	\end{itemize}
\end{example}
\begin{definition}[Essential equivalence] \label{essequ} An \defn{essential equivalence} of topological groupoids (or Lie groupoids, respectively) is a morphism $f \colon \G \to \HB$ of groupoids such that the following holds true:
	\begin{itemize}
		\item (essentially surjective) The map $\HB_1 \times_{s,f_0} \G_0 \to \HB_0$ induced by $t \colon \HB_1 \to \HB_0$ on the first factor is a surjection admitting local sections (or surjective submersion, respectively).
		\item (fully faithful) The square \[\begin{tikzcd}[sep=large]
		\G_1 \arrow{r}{f_1} \arrow{d}[left]{(s,t)} & \HB_1 \arrow{d}{(s,t)} \\
		\G_0 \times \G_0 \arrow{r}{(f_0,f_0)} & \HB_0 \times \HB_0
		\end{tikzcd} \] is a pullback in $\Top$ (or in $\Mfld$, respectively).
	\end{itemize}
	We say that two topological groupoids (or Lie groupoids, respectively) are \defn{essentially equivalent} if there is a zig-zag of essential equivalences between them. This is often also called \defn{Morita equivalent}.
\end{definition}
\begin{definition}[Stabilizer group]
	Let $\G$ be a topological groupoid and let $x$ be a point in $\G_0$. We write \[\G_x=(s,t)^{-1}(x,x)\] for the \defn{stabilizer group} of $x$. This becomes a topological group with the subspace topology.
\end{definition}
\begin{definition}[Proper, foliation groupoid] Let $\G$ be a Lie groupoid. We say that $\G$ is \defn{proper} if the map $(s,t) \colon \G_1 \to \G_0 \times \G_0$ is proper. We say that $\G$ is a \defn{foliation groupoid} if $\G_x$ is discrete for all $x \in \G_0$.
\end{definition}
\begin{lemma}
	Properness and being a foliation groupoid are invariant under essential equivalences.
\end{lemma}
\begin{proof}
	It follows immediately from the fully faithfulness that essential equivalences induce isomorphism on stabilizer groups. This proves that being a foliation groupoid is invariant under essential equivalence. The statement for properness is proven in \cite[Prop. 5.26]{mm03}.
\end{proof}
\begin{definition}[Orbifold] \label{deforbifold}
	An \defn{orbifold} is a proper foliation groupoid. We write $\Orbfld$ for the full subcategory of orbifolds in $\LieGrpd$.
\end{definition}
The category of orbifolds is usually defined as the categorical localization of the category of proper foliation groupoids at the essential equivalences. There are several explicit construction for this. A brief explanation using a calculus of fractions can be found in \cite[Sec. 2]{moe02}. A more detailed description which also explains how to invert the category as a $2$-category can be found in \cite[Sec. 3]{ler08}. 

However, we do not follow this idea. We will consider functors which will not necessarily invert essential equivalence but only send them to some kind of weak equivalence. Having this in mind, the category of proper foliation groupoids should be seen as the category of orbifolds and the class of essential equivalence is a class of weak equivalences.
\begin{remark} \label{lievstop}
	Any essential equivalence between proper Lie groupoids is an essential equivalence of underlying topological groupoids. The essentially surjectiveness condition holds true because any submersions admits local sections. Let $P$ denote the topological pullback of the diagram in the fully faithfulness condition. The induced map $\G_1 \to P$ is bijective and continuous, so we need to check that its inverse is continuous, too. We may check this locally at a point $x$ in the codomain $P$. Let $K$ be a compact neighborhood of the image of $x$ under the structure map $P \to \G_0 \times \G_0$. The restriction of the map $\G_1 \to P$ to the preimage of $K$ under $(s,t) \colon \G_1 \to \G_0 \times \G_0$ and the preimage of $K$ under the structure map $P \to \G_0 \times \G_0$ is a bijective map from a compact space to a Hausdorff space and hence a homeomorphism. 
\end{remark}
\begin{example}[Global quotient orbifold]
	Let $M$ be a manifold together with an almost free action of a compact Lie group $G$. The action groupoid $G \ltimes M$ is an orbifold and is called a \defn{global quotient orbifold}. 
\end{example}
A generalization of \Cref{efforbaregq} for general orbifolds as defined in this section was proven by Pardon.
\begin{theorem}[{\cite[Cor 1.3]{par20}}] \label{gqc}
	Every orbifold $\G$ is essentially equivalent to $G \ltimes M$ for some manifold $M$ together with an almost free action of a compact Lie group~$G$ on~$M$.
\end{theorem}
\begin{remark} \label{grpdtoeff}
	We briefly discuss how proper foliation groupoids may be seen as a generalization of effective orbifolds and refer the reader to \cite[Sec 5.6]{mm03} for more details. We want to associate a proper foliation groupoid to an effective orbifold $\chi$. Having \Cref{efforbaregq} in mind, an obvious choice, which is also often used, is the action groupoid~$\OG(n) \ltimes \Fr(\chi)$.
	
	There is another construction, which is more geometric and directly uses the charts of~$\chi$ and not the frame bundle construction. It can be found in \cite[Sec. 5.6]{mm03}. Both constructions lead to essentially equivalent orbifolds, so we may use the first one.  
	
	It can be shown that the associated Lie groupoids are essentially equivalent if and only if the original effective orbifolds are diffeomorphic, see \cite[Prop. 5.29(ii)]{mm03}.
	
	Conversely, there is also a good criteria for checking whether an orbifold comes from an effective orbifold (up to essential equivalence), see \cite[Thm. 5.32]{mm03}.
	
	
	
\end{remark}
For $\G$ a topological groupoid, we write $|\G|$ for the geometric realization of the nerve of $\G$ and we write $\|\G\|$ for the fat geometric realization of the nerve of $\G$.
\begin{proposition} \label{productfgr} 
	The geometric realization commutes with products, i.e. there is a natural homeomorphism \[|\G| \times |\HB| \cong |\G \times \HB|\] for any two topological groupoids $\G$ and $\HB$.
\end{proposition}
\begin{proof}
	The nerve functor commutes with products and so does the geometric realization, see \cite[Thm. 11.5]{may06}.
\end{proof}
\begin{proposition} \label{fgree}
	Let $f \colon \G \to \HB$ be an essential equivalence of topological groupoids. Then the induced map $\|f\| \colon \|\G\| \to \|\HB\|$ is a weak equivalence.
\end{proposition}
\begin{proof}
	This is mentioned as a corollary of \cite[Thm. 1.1]{noo12}. 
\end{proof}
The homotopy type of an orbifold is usually defined by using the fat geometric realization. 
\begin{definition}[Orbifold homotopy groups]\label{homotopygrouporbfld}
	Let $\G$ be an orbifold and let $k \ge 0$ be an integer. The $k$-th \defn{orbifold homotopy group} of $\G$ is defined as \[\pi_k(\G) = \pi_k(\|\G\|) \,. \]
\end{definition}
\Cref{fgree} implies that essential equivalences of orbifolds induce isomorphisms on all orbifold homotopy groups. 
\section{Orbispaces} \label{orbispaces}
The term \enquote{orbispace} is not used consistently in the literature. Orbispaces are always supposed to model spaces which locally arise as quotients of group actions. However, we want to warn the reader that the different models used in different areas of mathematics differ crucially. 
We will not use the term \enquote{orbispace} for well-defined mathematical objects but rather as a name for different approaches to the idea described above. 

We discuss two definitions: $\LI$-spaces as defined by Schwede \cite{sch19} and $\Orb$-spaces as defined by Gepner and Henriques \cite{gh07}. Moreover, we explain how to compare these two constructions by recalling a result from Körschgen's paper \cite{kor16}.
\subsection{\texorpdfstring{$\LI$}{L}-spaces} \label{Liorb}
We recall the definition of the topological monoid $\LI$, which contains every compact Lie group as a subgroup. A space with an action of $\LI$ then gives rise to a $G$-space for an arbitrary compact Lie group $G$ by restricting the action of $\LI$ to $G$.

The following is mainly taken from \cite{sch19}. We require all representations to be orthogonal representations on inner product spaces.
\begin{definition}[$\LI$-space]
	Let $\LI=\LB(\R^\infty,\R^\infty)$ be the space of linear isometric self-embeddings of $\R^\infty$ together with the standard inner product that carries the compactly-generated function space topology as described in \cite[Sec. 1]{sch19}. The space $\LI$ together with the composition of self-embeddings is a topological monoid.

	An \defn{$\LI$-space} $A$ is a space together with a continuous action of the topological monoid~$\LI$. A \defn{morphism} of $\LI$-spaces is an $\LI$-equivariant continuous map of the underlying spaces. We write $\LI\dash\Top$ for the category of $\LI$-spaces. 
\end{definition}
\begin{definition}[Complete $G$-universe]
	Let $G$ be a compact Lie group. A \defn{complete $G$-universe} is a $G$-representation $\mathcal U_G$ of countably infinite dimension such that every finite-dimensional $G$-representation is isomorphic to a subrepresentation of $\mathcal U_G$. 
\end{definition}
\begin{definition}[Universal subgroup] \label{universalsubgroup}
	A subgroup $G$ of $\LI$ is \defn{universal} if it is compact, admits the structure of a Lie group and $\R^\infty$ is a complete $G$-universe with the $G$-action on $\R^\infty$ induced by the embedding $G \longhookrightarrow \LI=\LB(\R^\infty,\R^\infty)$.
\end{definition}
\begin{lemma}[{\cite[Prop. 1.5]{sch19}}] \label{allareuniversal}
	Every compact Lie group is isomorphic to a universal subgroup of $\LI$. Moreover, every isomorphism between universal subgroups of $\LI$ is given by conjugation with an invertible element of $\LI$.
\end{lemma}
Let $G$ be a universal subgroup of $\LI$. 
The forgetful functor $\LI\dash\Top \to G\dash\Top$ admits a left adjoint. For a $G$-space $A$ this left adjoint is given by $\LI \times_G A$ where $G$ acts on $\LI$ by pre-composition with an inverse.
\begin{definition}[Global equivalence]
	A map $f \colon A \to B$ of $\LI$-spaces is called a \defn{global equivalence} if $f^G \colon A^G \to B^G$ is a weak equivalence for all universal subgroups $G$ of $\LI$.
\end{definition}
\begin{remark}[Model structures on $\LI\dash\Top$] 
	The global equivalences are the weak equivalences in two model structures on $\LI\dash\Top$. 
	
	The \defn{universal projective model structure} is the model structure where a map $f \colon X \to Y$ is a fibration if the map $f^G \colon X^G \to Y^G$ is a Serre fibration for all universal subgroups $G$ of $\LI$, see \cite[Prop. 1.11]{sch19}.
	
	The \defn{global model structure} on $\LI\dash\Top$ arises as a left Bousfield localization at the global equivalences on the projective model structure with respect to all compact Lie subgroups, see \cite[Thm. 1.20]{sch19} for more details. 
	
	A comparison between these two model structures can be found in \cite[Rem. 2.7]{sch19}.
\end{remark}
There is an analog of Elmendorf's theorem \cite[Thm. 1]{elm83} for $\LI$-spaces. Roughly speaking, the fixed point spaces for universal subgroups and certain maps between them determine the global homotopy type of an $\LI$-space. 
\begin{definition}[$\Ogl$-space]
	Let $\Ogl$ be the topological category with 
	\begin{itemize}
		\item objects all universal subgroups of $\LI$,
		\item the morphism space between two universal subgroups $G$ and $H$ of $\LI$ given by \[ \Ogl(H,G) = \map^\LI(\LI/H,\LI/G) \cong \left(\LI/G \right)^H \,. \]
	\end{itemize}
	The composition is defined by composition of $\LI$-spaces.
	An \defn{$\Ogl$-space} is a continuous functor \[(\Ogl)^{\op} \longrightarrow \Top \,.\] A \defn{morphism} of $\Ogl$-spaces is a continuous natural transformation. 
	We write $\Top_{\Ogl}$ for the category of $\Ogl$-spaces.
\end{definition}
\begin{remark}[Projective model structure on $\Top_{\Ogl}$]
	We endow $\Top_{\Ogl}$ with the \defn{projective model structure}, i.e. the weak equivalences (or fibrations, respectively) are those maps which are levelwise weak equivalences (or Serre fibration, respectively). The projective model structure on a topologically enriched functor category is established in \cite[Sec. 5]{pia91}. 
\end{remark}
\begin{definition}[Fixed point functor]
	There is a \defn{fixed point functor} \[\Phi \colon \LI\dash\Top \longrightarrow \Top_{\Ogl} \] with value at an $\LI$-space $X$ given by \[ \Phi(X)(K)=\map^{\LI}(\LI/K,X) \cong X^K \] for a universal subgroup $K$ of $\LI$. The structure maps act by pre-composition and functoriality is given by post-composition.
\end{definition}
\begin{theorem} \label{elmendorfLI}
	The fixed point functor $\Phi$ admits a left adjoint \[\Lambda \colon \Top_{\Ogl} \longrightarrow \LI\dash\Top \,. \] Moreover, this pair is a Quillen equivalence between $\Top_{\Ogl}$ together with the projective model structure and $\LI\dash\Top$ together with the universal projective model structure.
\end{theorem}
\begin{proof}
	This is proven in \cite[Thm. 2.5]{sch19}. Moreover, \cite[Rem. 2.7]{sch19} is an additional note on the model structures.
\end{proof}
\subsection{\texorpdfstring{$\Orb$}{Orb}-spaces} \label{orborb}
The idea of $\Orb$-spaces is also inspired by Elmendorf's theorem and hence quite similar to $\Ogl$-spaces. We are just using another indexing category.
The following is taken from \cite[Sec. 4]{gh07}.
\begin{definition}[$\Orb$-space]
	Let $\Orb$ be the topological category with 
	\begin{itemize}
		\item objects all universal subgroups of $\LI$,
		\item the morphism space between two universal subgroups $G$ and $H$ of $\LI$ given by \[| \gmap(\B H, \B G) | \cong \map(H,G) \times_G \EC G \,.\] Here, $\B G$ denotes the groupoid associated to the group $G$, see \Cref{examplegroupoids} and~$\gmap$ is the mapping groupoid, see \Cref{deftopgrpd}.
	\end{itemize}
	Let $G,H$ and $K$ be universal subgroups of $\LI$. The composition map is defined as the composition of the natural homeomorphism \[ | \gmap(\B H,\B K) | \times | \gmap(\B G,\B H) | \cong | \gmap(\B H, \B K) \times  \gmap(\B G, \B H) |\]
	from \Cref{productfgr} and \[ | c | \colon  |\gmap(\B H,\B K) \times \gmap(\B G, \B H) | \longrightarrow | \gmap(\B G,\B K) |\] where $c$ is the composition in $\TopGrpd$. The identity of $G$ in $\Orb$ is the $0$-simplex in~$| \gmap(\B G, \B G) |$ representing the identity.
	An \defn{$\Orb$-space} is a continuous functor \[\Orb^{\op} \longrightarrow \Top\,.\]
	A \defn{morphism} of $\Orb$-spaces is a continuous natural transformation. We write $\Top_{\Orb}$ for the category of $\Orb$-spaces. 
\end{definition}
We could also use all compact Lie groups as objects. The resulting category will be Dwyer-Kan equivalent to $\Orb$ as defined above. However, our definition makes it easier to compare $\Orb$-spaces to $\LI$-spaces.
\begin{remark}
	Gepner and Henriques use the fat geometric realization in their definition of the category $\Orb$. The fat geometric realization does not commute with products on the nose, i.e. the map $\|\G \times \HB\| \to \|\G\| \times \|\HB\|$ is just a weak equivalence. Gepner and Henriques assume the existence of natural retracts of these maps, see \cite[Rem. 2.23]{gh07}. It is not clear to the author why such retracts exist. 
\end{remark}
\begin{definition}[Weak equivalence]
	A morphism of $\Orb$-spaces $f \colon X \to Y$ is a \defn{weak equivalence} if $f(K) \colon X(K) \to Y(K)$ is a weak equivalence for all universal subgroups $K$ of $\LI$.
\end{definition}
\begin{remark}[Projective model structure on $\Top_{\Orb}$]
	The class of weak equivalences are the weak equivalences of the \defn{projective model structure} on $\Orb$-spaces. The fibrations are levelwise Serre fibrations. The projective model structure on a topologically enriched functor category is established in \cite[Sec. 5]{pia91}. 
\end{remark}
\begin{definition} \label{Rfunctor}
	We recall the functor \[F \colon \TopGrpd \longrightarrow \Top_{\Orb}  \] from \cite[Sec. 4.2]{gh07}. 
	On objects it is given by \[F(\G)(K) = | \gmap(\B K,\G) |\] for a topological groupoid $\G$ and a universal subgroup $K$ of $\LI$. The map on morphism spaces \[ | \gmap(\B K,\B H) | \longrightarrow \map\left(|\gmap(\B H,\G) | ,| \gmap(\B K, \G)|\right) \] for another universal subgroup $H$ of $\LI$ is defined as the adjoint of the map \[ |\gmap(\B K,\B H)| \times | \gmap(\B H,\G) | \longrightarrow |\gmap(\B K,\G)|  \] again using the composition of the homeomorphism from \Cref{productfgr} and the induced morphism of composition of groupoids. 
	
	A morphism between two groupoids $\G$ and $\HB$ induces by post-composition and the functoriality of the geometric realization a morphism between $F(\G)$ and $F(\HB)$.
\end{definition}
\subsection{Comparison between \texorpdfstring{$\LI$}{L}-spaces and \texorpdfstring{$\Orb$}{Orb}-spaces} \label{comparison}
Körschgen's paper \cite{kor16} provides a zig-zag of Quillen equivalences between $\Top_{\Ogl}$ and $\Top_{\Orb}$. We briefly recall some important definitions and results.

Körschgen is working with right $G$-spaces in his definition of the category~$\Orb^\prime$ while we prefer to work with left $G$-spaces only. To be consistent, we consider right $G$-spaces as left $G$-spaces via $g.x = x.g^{-1}$ for $x \in A$ and $g \in G$ where $A$ is a right~$G$-spaces.
\begin{notation}[Graph subgroup]
	Let $H$ and $G$ be two topological groups. Let $L$ be a closed subgroup of $H$ and let $\alpha \colon L \to G$ be a continuous group homomorphism. The \defn{graph subgroup} of $\alpha$ is the subgroup of $H \times G$ of all elements of the form $(h,\alpha(h))$ for $h \in L$. For an $(H \times G)$-space $A$ we write $A^\alpha$ for the fixed point space of $A$ of the graph subgroup of $\alpha$.
\end{notation}
\begin{definition}[$\Orb^\prime$-spaces] \label{Orbprime}
	Let $G$ and $H$ be universal subgroups of $\LI$.
	Let $H\times G$ act on $\LI$ via $(h,g).\varphi = h\varphi g^{-1}$. Define the space \[\tilde \EC(H,G) = \setdef{(\alpha,\varphi) \in \map(H,G) \times \LI}{\varphi \in \LI^\alpha} \,. \] The group $G$ acts on $\tilde \EC(H,G)$ via $g.(\alpha,\varphi)=(\alpha^g,\varphi g^{-1})$ where $(\alpha^g)(h)=g\alpha(h)g^{-1}$ for~$h \in H$.
	We define the category $\Orb^\prime$ with objects all universal subgroups of $\LI$ and morphism space \[\Orb^\prime(H,G) = \tilde E(H,G) \times_G \EC G \,. \] for two universal subgroups $G$ and $H$ of $\LI$. The composition is defined in \cite[Sec. 3.1]{kor16}.
	
	The category $\Top_{\Orb^{\prime}}$ of \defn{$\Orb^\prime$-spaces} is the category of continuous functors \[(\Orb^\prime)^{\op} \longrightarrow \Top\] together with continuous natural transformations. We equip this category with the projective model structure.
\end{definition}
It is proven in \cite[Prop. 2.18]{kor16} that the projection map $\tilde E(H,G) \to \map(H,G)$ gives rise to a weak equivalence \[\Orb^\prime(H,G) = \tilde E(H,G) \times_G \EC G \longrightarrow \map(H,G) \times_G \EC G \cong \Orb(H,G) \,. \] This map induces a functor $f_1 \colon \Orb^\prime \to \Orb$. 
On the other hand, the canonical map \[\Orb^\prime(H,G) = \tilde E(H,G) \times_G \EC G  \longrightarrow \tilde E(H,G)/G \] is a weak equivalence by \cite[Prop. 2.22]{kor16} and gives rise to a functor $f_2 \colon \Orb^\prime \to \Ogl$ because \begin{align*}
\tilde E(H,G)/G & \longrightarrow \left(\LI/G \right)^H \cong \Ogl(H,G) \,, \\ 
[\alpha,\varphi] & \longmapsto [\varphi]
\end{align*}
is well-defined and a homeomorphism by \cite[Prop. 2.20]{kor16}.
Finally, it is shown that \begin{align*}(f_1)^\ast \colon & \Top_{\Orb} \longrightarrow \Top_{\Orb^\prime}\shortintertext{and} (f_2)^\ast \colon & \Top_{\Ogl} \longrightarrow \Top_{\Orb^\prime}\end{align*} are left Quillen functors and both part of a Quillen equivalence. Their right adjoints will be denoted by $(f_1)_{!}$ and $(f_2)_{!}$, respectively.

\section{Orbifolds as Orbispaces}
We will give two constructions to associate an orbispace to an orbifold. The first one in \Cref{effli} does only work for effective orbifolds but it is easier to compute in practice. In \Cref{grouporborb} we define a functor $\Orbfld \to \Top_{\Orb}$ which is extended to a functor $\Orbfld \to \LI\dash\Top$ in \Cref{comparison2} but it is difficult to compute the value of this functor at an orbifold because it uses a cofibrant replacement. 
\subsection{Effective Orbifolds as \texorpdfstring{$\LI$}{L}-spaces} \label{effli}
Every effective orbifold is diffeomorphic to an effective global quotient orbifold by \Cref{efforbaregq}. Therefore, we may define the associated $\LI$-space to an effective orbifold as the corresponding model in $\LI$-spaces for the quotient.
\begin{definition}[Assignment $\effli$] \label{functorLI}
	Let $\chi$ be an effective orbifold of dimension $n$. Define \[\effli(\chi)=\LI \times_{\OG(n)} \Fr(\chi) \] for an embedding of $\OG(n)$ in $\LI$ as a universal subgroup.
\end{definition}
This definition depends on a choice of a Riemannian metric on $\chi$ and on the embedding of $\OG(n)$ into $\LI$. By \Cref{framebundleindependent} and \Cref{allareuniversal}, two differenct choices lead to $\LI$-equiviarently homeomorphic $\LI$-spaces. 

An effective global quotient orbifold $M \gq G$ should be represented by the $\LI$-space $\LI \times_G M$, up to global equivalence. 
\begin{proposition} \label{RC1m0}
	Let $G$ be universal subgroup of $\LI$ which acts almost freely and effectively on slices on a manifold $M$. Then the $\LI$-space $\effli(M \gq G)$ is $\LI$-equivariantly homotopy equivalent to $\LI \times_G M$.
\end{proposition}
We first need to prove two lemmas.
\begin{lemma} \label{fexLhom}
	Let $G$ and $H$ be universal subgroups of $\LI$. Let $A$ be an $(H \times G)$-CW complex such that the restricted $G$-action is free. There exists an $(H \times G)$-equivariant continuous map $a \colon A \to \LI$ where $H \times G$ acts on $\LI$ via $(h,g).\varphi=h\varphi g^{-1}$ for $g \in G, h \in H$ and~$\varphi \in \LI$.
\end{lemma}
\begin{proof}
	By \cite[Prop. A.10]{sch19} and \cite[Prop. 1.1.26(i)]{sch18}, $\LI$ is $(H \times G)$-equivariantly homotopy equivalent to a universal space for the class of graph subgroups, that are all subgroups $K$ of $H \times G$ such that $K \cap (\{1\} \times G)=\{(1,1)\}$. This includes in particular all stabilizer groups of $A$ because the restricted $G$-action was assumed to be free. The existence of the map $a$ follows from \cite[Prop. B.11(i)]{sch18}.
\end{proof}
\begin{lemma} \label{Lhom}
	Let $G$ and $H$ be universal subgroups of $\LI$. Let $A$ be an $(H \times G)$-CW complex such that the restricted actions of $G$ and $H$ on $A$ are free. Then $\LI\times_GA/H$ is $\LI$-equivariantly homotopy equivalent to $\LI\times_HA/G$.
\end{lemma}
\begin{proof} Throughout this proof, $x$ denotes a point in $A$, $g$ an element of $G$ and $h$ and element of $H$.
	Using \Cref{fexLhom}, choose an $(H \times G)$-equivariant continuous map $a \colon A \to \LI$. Consider the map \begin{align*}
	\overline a \colon \LI \times A/G & \longrightarrow \LI\times_G A/H  \,,\\
	(\varphi,[x]) & \longmapsto [\varphi a(x),[x]] \,.
	\end{align*}
	This map is independent of the choice of the representative of $[x]$ because \[[\varphi a(g.x),[g.x]]=[\varphi a(x)g^{-1},g.[x]]=[\varphi a(x),[x]]  \,. \] Furthermore, $\overline a$ induces a map \[\hat a \colon \LI \times_H A/G \longrightarrow \LI \times_G A/H\] because \[ \overline a(\varphi h^{-1},h.[x]) = [\varphi h^{-1} a(h.x),[h.x]] = [\varphi a(x),[x]] = \overline a(\varphi,x) \,. \]
	Similarly, we choose a $(G \times H)$-equivariant continuous map $b \colon A \to \LI$ and obtain a map \[\hat b \colon \LI \times_G A/H \longrightarrow \LI \times_H A/G \]which is defined analogously. 
	We prove that $\hat a$ and $\hat b$ are mutually inverse $\LI$-equivariant homotopy equivalences. We define a map $c \colon A \to \LI$ via $c(x)=b(x)a(x)$. We use that \[c((g,h).x)=gc(x)g^{-1} \,. \] Hence $c$ is $H$-invariant and $G$-equivariant for the diagonal $G$-action on $\LI$. Therefore,~$c$ factors through a $G$-equivariant continuous map $\tilde c \colon A/H \to \LI$ via the canonical projection~$A \to A/H$.
	We write $\hat c$ for $\hat a \circ \hat b$. Note that the following equality holds true: \[\hat c([\varphi,[x]]) = [\varphi \tilde c([x]),[x]] \,.\] The space $\LI$ is $G$-equivariantly contractible by \cite[Proposition 1.1.21]{sch18} together with \cite[Proposition A.10]{sch19}. Choose a $G$-equivariant homotopy \[k \colon \LI \times [0,1] \longrightarrow \LI\] such that $k(\varphi,0)=\varphi$ and $k(\varphi,1)=\id$. Let $t$ denote an element in~$[0,1]$. Define an $\LI$-equivariant homotopy \begin{align*}
	\left(\LI \times_{G} A/H \right) \times [0,1] & \longrightarrow \LI \times_G A/H \,,\\
	([\varphi,[x]],t) & \longmapsto \left[\varphi k(\hat c([x]),t),[x] \right]
	\end{align*} 
	from $\hat c = \hat a \circ \hat b$ to the identity. The well-definiteness can be shown by a similar computation as for $\hat a$ by using that $k(\hat c(-),t) \colon A/H \to \LI$ is $G$-equivariant for every fixed $t \in [0,1]$. 
	
	Similarly, $\hat b \circ \hat a$ is also $\LI$-equivariantly homotopic to the identity on $\LI \times_H A/G$.
\end{proof}
\begin{proof}[Proof of \Cref{RC1m0}]
	Let $k$ denote the integer $\dim M-\dim G$. We use that $\Fr(M \gq G) \cong \Fr_G(M)/G$ as explained in \Cref{framebundleofMG} and that $M \cong \Fr_G(M)/\OG(k)$ and the fact that $\Fr_G(M)$ admits the structure of a $(\OG(k)\times G)$-CW complex by \cite[Cor. 7.2]{ill83}. The statement comes down to \Cref{Lhom} with $A=\Fr_G(M)$ and $H=\OG(k)$. 
\end{proof}
\subsection{Orbifolds as \texorpdfstring{$\Orb$}{Orb}-spaces} \label{grouporborb}
The functor $F \colon \TopGrpd \to \Top_{\Orb}$  from \Cref{Rfunctor} gives rise to a functor from orbifolds to $\Orb$-spaces. 
\begin{definition}[Functor $\overline F$] \label{lielifunctor}
	Let \[\overline F \colon \Orbfld \longrightarrow \Top_{\Orb} \] be the composition of the full embedding $\Orbfld \hookrightarrow \LieGrpd$, the forgetful functor $\LieGrpd \to \TopGrpd$ and the functor $F \colon \TopGrpd \to \Top_{\Orb}$.
\end{definition}
We check in \Cref{overlineF} that $\overline F$ sends essential equivalences of orbifolds to weak equivalences of $\Orb$-spaces.
This is proven in several steps. In the upcoming \Cref{mapee}, it is shown that \[\|\gmap(\B K,-)\| \colon \TopGrpd \to \Top\] sends essential equivalences to weak equivalences whenever $K$ is topological group. Afterwards, we use Segal's criteria from \cite[Prop. A.1(iv)]{seg74} to show that this fat geometric realization is weakly equivalent to the geometric realization whenever we started with a Lie groupoid. The technical condition will be checked in \Cref{isgood}. 
\begin{proposition} \label{mapee}
	Let $K$ be a topological group. Let $f \colon \G \to \HB$ be an essential equivalence of topological groupoids. Then the induced map of groupoids \[f_\ast \colon \gmap(\B K, \G)\longrightarrow \gmap(\B K ,\HB)\] is again an essential equivalence.
\end{proposition}
The proof splits up into several lemmas.
Let \[p_\G \colon \gmap(\B K,\G) \longrightarrow \gmap(\B \tg,\G) \cong \G \]
be induced by the unique morphism $\B \tg \to \B K$ where $\tg$ denotes the trivial subgroup of $\LI$. 
\begin{lemma} \label{pb0}
	Let $K$ be a topological group and let $f\colon \G \to \HB$ be an essential equivalence of topological groupoids. Then the diagram
	\[\begin{tikzcd}[column sep=large]
	\gmap(\B K,\G)_0 \arrow{r}{(p_\G)_0} \arrow[swap]{d}{(f_\ast)_0} & \G_0 \arrow{d}{f_0} \\
	\gmap(\B K,\HB)_0 \arrow{r}{(p_\HB)_0} & \HB_0
	\end{tikzcd}\] is a pullback. 
\end{lemma}
\begin{proof}
	Recall that $\gmap(\B K, \HB)_0$ is topologized as a subset of $\map(K,\HB_1)$. Evaluation at $K$ hence induces a map $\gmap(\B K,\HB)_0 \times K \to \HB_1$. One can check on elements that the following diagram commutes: \[\begin{tikzcd}[sep=large]
	\gmap(\B K,\HB)_0 \times K \ar[r,"(p_\G)_0 \circ \pr_1"] \ar[d] & \HB_0 \ar[d,"\Delta_{\HB_0}"] & \G_0 \ar[d,"\Delta_{\G_0}"] \ar[l,swap,"f_0"] \\
	\HB_1 \ar[r,"{(s_{\HB},t_{\HB})}"] & \HB_0 \times \HB_0 & \G_0 \times \G_0 \ar[l,swap,"{(f_0,f_0)}"]
	\end{tikzcd} \]
	Using that \[\left(\gmap(\B K,\HB)_0 \times K \right) \times_{\HB_0} \G_0 \cong \left(\gmap(\B K,\HB)_0 \times_{\HB_0} \G_0 \right) \times K \,, \] we obtain an induced map on pullbacks \[\left(\gmap(\B K,\HB)_0 \times_{\HB_0} \G_0 \right) \times K \longrightarrow \HB_1 \times_{\HB_0 \times \HB_0} \left(\G_0 \times \G_0 \right) \cong \G_1 \] where the last isomorphism uses that $f$ is an essential equivalence. One can now check on elements that the adjoint map \[\gmap(\B K,\HB)_0 \times_{\HB_0} \G_0  \longrightarrow \map(K,\G_1) \] takes values in $\gmap(\B K,\G)_0$ and that this map is inverse to the map \[\gmap(\B K,\G)_0 \longrightarrow \gmap(\B K,\HB)_0 \times_{\HB_0} \G_0 \] induced by $(f_\ast)_0$ and $(p_\G)_0$.
\end{proof}
We write $s_{\B K,\G}$ and $t_{\B K,\G}$ for the source and target map of the topological groupoid $\gmap(\B K,\G)$.
\begin{lemma} \label{pbcm}
	Let $K$ be a topological group and let $\G$ be a topological groupoid. The diagram \[\begin{tikzcd}[sep = large]
	\gmap(\B K,\G)_1 \arrow{r}{s_{\B K,\G}} \arrow{d}[left]{(p_\G)_1} & \gmap(\B K,\G)_0 \arrow{d}{(p_\G)_0} \\
	\G_1 \arrow{r}{s_\G} & \G_0
	\end{tikzcd} \] is a pullback. 
\end{lemma}
\begin{proof}
	Similarly as in the proof of \Cref{pb0}, evaluation at an element of $K$ and diagonal maps induce a map \[\gmap(\B K,\G)_0 \times_{\G_0} \G_1 \times K \longrightarrow \G_1 \times_{\G_0 \times \G_0} \left(\G_1 \times \G_1\right) \,. \] There is a homeomorphism \[\G_1 \times_{\G_0 \times \G_0} (\G_1 \times \G_1) \cong \G_1 \times_{\G_0} \G_1 \times_{\G_0} \G_1 \] sending $(g_1,(g_2,g_3))$ to $(g_3,g_1,g_2^{-1})$. Composition in $\G$ hence induces a map \[\gmap(\B K,\G)_0 \times_{\G_0}\G_1 \times K \longrightarrow \G_1 \,.\] Consider the map \[\gmap(\B K,\G)_0 \times_{\G_0}\G_1 \longrightarrow \map(K,\G_1) \times \map(K,\G_1) \times \map(\ast, \G_1) \] which is induced by the inclusion $\gmap(\B K,\G)_0 \hookrightarrow \map(K,\G_1)$, the adjoint of the above map and the homeomorphism $\G_1 \cong \map(\ast,\G_1)$. One can check on elements that this map takes values in $\gmap(\B K,\G)_1$ and that it is inverse to the map from $\gmap(\B K,\G)_1$ to the pullback $\gmap(\B K,\G)_0 \times_{\G_0} \G_1$ induced by $s_{\B K,\G}$ and $(p_\G)_1$.
\end{proof}
\begin{remark} \label{pbcmt}
	Similarly one shows that the diagram in \Cref{pbcm} is a pullback when replacing the source maps with target maps.
\end{remark}
\begin{proof}[Proof of \Cref{mapee}]
	We first prove that $f_\ast$ is essentially surjective as in \Cref{essequ} of essential equivalences. Let \[P=\gmap(\B K,\HB)_1 \times_{{\gmap(\B K,\HB)}_0} \gmap(\B K,\G)_0\] denote the pullback along $s_{\B K,\HB}$ and $(f_\ast)_0$. Consider the commutative diagram \[\begin{tikzcd}[sep = large]
	& \gmap(\B K,\G)_0 \ar[rr,"(p_\G)_0"] \ar[dd,near end,swap,"(f_\ast)_0"] & & \G_0 \ar[dd,"f_0"] \\
	P \ar[ur] \ar[rr,crossing over,"\varphi" near end] \ar[dd] & & \HB_1 \times_{\HB_0} \G_0 \ar[ur] \\
	& \gmap(\B K,\HB)_0 \ar[rr,"(p_\HB)_0" near start] & & \HB_0 \\
	\gmap(\B K,\HB)_1 \ar[rr,swap,"(p_{\HB})_1"] \ar[ur,"s_{\B K,\HB}"] & & \HB_1 \ar[ur,"s_\HB" swap] \ar[from=uu,crossing over]
	\end{tikzcd} \]
	where the unlabeled maps out of the pullbacks are the structure maps and $\varphi$ is the map induced by the functoriality of pullbacks. 
	
	The squares on the left and on the right are pullbacks by definition. The square in the back also is a pullback by \Cref{pb0}. Hence, the square in the front is a pullback.
	Consider the following commutative diagram: \[\begin{tikzcd}[sep = large]
	P \ar[d,"\varphi"] \ar[r] & \gmap(\B K,\HB)_1 \ar[d,"(p_\HB)_1"] \ar[r,"t_{\B K,\HB}"] & \gmap(\B K,\HB)_0\ar[d,"(p_\HB)_0"]  \\
	\HB_1 \times_{\HB_0} \G_0 \ar[r] & \HB_1 \ar[r,"t_\HB"] & \HB_0
	\end{tikzcd} \]
	We have just shown that the left square is a pullback and the right square is a pullback by \Cref{pbcmt}. Therefore, the outer square also is a pullback. The composition of the lower maps is a surjection admitting local sections by assumption on $f$. Now we can conclude that the composition of the maps in the top row also is a surjection that admits local sections.
	
	To see that $f_\ast$ is fully faithful as in \Cref{essequ}, consider the following commutative diagram:
	\[ \begin{tikzcd}[column sep = small,row sep = large]
	& \gmap(\B K,\G)_0 \ar[rr,"(f_\ast)_0"] \ar[dd,"(p_\G)_0" near end, swap] & & \gmap(\B K,\HB)_0 \ar[dd,"(p_\HB)_0"] \\
	\gmap(\B K,\G)_1 \ar[ur,"s_{\B K,\G}"] \ar[rr,"(f_{\ast})_1" near end, crossing over] \ar[dd,swap,"(p_\G)_1"] & & \gmap(\B K,\HB)_1 \ar[ur,"s_{\B K,\HB}"]\\
	& \G_0 \ar[rr,"f_0" near end] & & \HB_0 \\
	\G_1 \ar[ur,"s_\G"] \ar[rr,"f_1"] & & \HB_1 \ar[ur,"s_\HB" swap] \ar[from=uu,"(p_{\HB})_1" near end, swap, crossing over] 
	\end{tikzcd}\] 
	The square in the back is a pullback by \Cref{pb0}. The squares on the left and on the right are pullbacks by \Cref{pbcm}. Therefore, the square in the front is also a pullback. 
	Consider the following commutative diagram: \[\begin{tikzcd}[row sep=large]
	& \G_1 \ar[rr,"f_1"] \ar[dd,"{(s_\G,t_\G)}" near end, swap] & & \HB_1 \ar[dd,"{(s_\HB,t_\HB)}"] \\
	\gmap(\B K,\G)_1 \ar[dd,"{(s_{\B K,\G},t_{\B K,\G})}" swap] \ar[rr,"{(f_\ast)_1}" near end, crossing over] \ar[ur,"(p_\G)_1"] & & \gmap(\B K,\HB)_1  \ar[ur,swap,"(p_\HB)_1"] \\
	& (\G_0)^2 \ar[rr,"{(f_0)^2}" near end] & & (\HB_0)^2 \\
	\left(\gmap(\B K,\G)_0 \right)^2  \ar[ur,"{((p_\G)_0)^2}"] \ar[rr,swap,"{((f_\ast)_0)^2}"] & & \left(\gmap(\B K,\HB)_0\right)^2 \ar[ur,"{((p_\HB)_0)^2}" swap] \ar[from=uu,"{(s_{\B K,\HB},t_{\B K,\HB})}", near end, swap, crossing over]
	\end{tikzcd}\]
	The square in the back is a pullback by assumption on $f$. The lower square is a pullback by \Cref{pb0}. We have just shown that the square on the top is also a pullback. All in all, we can conclude that the square in the front also is a pullback. 
\end{proof}
\Cref{mapee} and \Cref{fgree} imply that the induced map \[\|\gmap(\B K,\G) \| \longrightarrow \|\gmap(\B K,\HB)\|\] is a weak equivalence for every essential equivalence $f \colon \G \to \HB$. We want to conclude the same for the geometric realization instead of the fat geometric realization. We prove this for the case where $\G$ and $\HB$ are Lie groupoids. 
\begin{lemma} \label{pbiscc}
	Let $M$ and $N$ be manifolds and let $A$ be a space. Let $a \colon A \to N$ be a continuous map and let $r \colon M \to N$ be a smooth map with a smooth section $c \colon N \to M$. The section~$c^\prime \colon A \to A \times_N M$ induced from the identity $A \to A$ and the composition~$c \circ a \colon A \to M$ is a Hurewicz cofibration. 
	\[ \begin{tikzcd}[sep = large]
	A \times_N M \ar[r]\ar[d]& M \ar[d,swap,"r"] \\
	A \ar[r,"a"] \ar[u,bend right=50,swap,dashed,"c^\prime"] & N \ar[u,bend right=50,swap,"c"] 
	\end{tikzcd}\]
\end{lemma}
\begin{proof}
	We use the following characterization of Hurewicz cofibrations from \cite[Thm. 1.7]{de72}: A subspace inclusion $A \hookrightarrow A \times_N M$ is a Hurewicz cofibration if it is a strong deformation retract of a halo $U$ in $A \times_N M$, that is a subset $U=f^{-1}([0,1))$ for a map $f \colon A \times_N M \to [0,1]$ such that $f^{-1}(0)=A$.
	
	The map $c \colon N \to M$ is a smooth section and hence both a closed embedding and an immersion. Choose a tubular neighborhood $U$ of $N$ in $M$. The subset $U$ is a halo of $N$ as can be seen by choosing a metric on the manifold $M$. 
	Let $\tilde U$ be the preimage of $U$ in $A \times_N M$ under the structure map $A \times_N M \to M$. Then $\tilde U$ is a halo of $A$ in $A \times_N M$. 
	
	The restriction of the structure map $ A \times_N M \to A$ to $\tilde U$ is a pullback of a vector bundle and hence again a vector bundle. The map $c^\prime \colon A \to A \times_N M$ takes values in $\tilde U$ and is the $0$-section of the pullback vector bundle and hence a strong deformation retract of the halo $\tilde U$. 
\end{proof}
\begin{lemma} \label{isgood}
	Let $\G$ be a Lie groupoid and let $K$ be a Lie group. The nerve $\nerve(\gmap(\B K,\G))$ is a good simplicial space, 
	i.e. all degeneracy maps are Hurewicz cofibrations. 
\end{lemma}
\begin{proof}
	We write $M$ for $\nerve(\gmap(\B K,\G))$. We use the indexing for face and degeneracy maps from \cite[Def. 1.1]{may67}. It is enough to prove that for all integers $1 \le i \le n$ the diagram \[ \begin{tikzcd}[sep = large]M_{n} \ar[r,"(p_\G)_n"] \ar[d,swap,"d_i^{M}"] & \G_n \ar[d,"d_i^\G"]  \\
	M_{n-1} \ar[r,"(p_\G)_{n-1}"] & \G_{n-1}\end{tikzcd} \] is a pullback for the face maps $d_i^\G \colon \G_n \to \G_{n-1}$ and $d_i^{M} \colon M_n \to {M}_{n-1}$. The degeneracy map $\sigma_{i-1}^\G \colon \G_{n-1} \to \G_n$ is a section of $d_i^\G$ and induces the degeneracy map~$\sigma_{i-1}^{M} \colon M_{n-1} \to M_n$ which is then a Hurewicz cofibration by \Cref{pbiscc}. 
	
	We prove this by induction on $n$. For $n=1$, we can have $i=1$ only and this is \Cref{pbcm} because $d_1$ is the source map. Let $n \ge 2$ and $1 \le i \le n$. Consider the commutative diagram \begin{equation} \label{properhelp} \begin{tikzcd}[column sep = large, row sep = huge]
	M_n \ar[r,"{d_i^{M}}"] \ar[d,swap,"{(p_\G)_n}"] & M_{n-1} \ar[r,"s \circ \pr_1"] \ar[d,swap,"{(p_\G)_{n-1}}"] & M_0 \ar[d,"{(p_\G)_0}"] \\
	\G_n \ar[r,"{d_i^{\G}}"] & \G_{n-1} \ar[r,"s\circ \pr_1"]  & \G_0 
	\end{tikzcd}\end{equation}
	where $s \circ \pr_1$ is the source of the the first component of $M_{n-1}$ (or $\G_{n-1}$, respectively), that is the map induced by $[0] \mapsto [n-1], 0 \mapsto 0$.
	
	The  square on the right is a pullback by induction (and moreover using that $s \circ \pr_1$ is a composition of face maps). The square on the left is a pullback for $i=n$ because~$d_n$ is the projection onto the first $n-1$ factors and hence the diagram \[\begin{tikzcd}[row sep = large, column sep = huge]
	& M_1 \ar[dd,near start,"{(p_\G)_1}"] \ar[rr,"s"] & & M_0 \ar[dd,"{(p_\G)_0}"] \\
	M_n \ar[rr,near end,swap,crossing over,"d_n^{M}"] \ar[dd,swap,"(p_\G)_n"] \ar[ur," \pr_n"] & & M_{n-1} \ar[ur,"t\circ\pr_{n-1}"] \\
	& \G_1 \ar[rr,near end,"s"] & & \G_0 \\
	\G_n \ar[rr,swap,"d_n^\G"] \ar[ur,"\pr_n"] & & \G_{n-1} \ar[ur,swap,"t \circ \pr_{n-1}"] \ar[from=uu,,near end, swap,crossing over,"(p_\G)_{n-1}"]
	\end{tikzcd} \]
	commutes. The square on the top and on the bottom are pullbacks by definition. The square in the back is a pullback by \Cref{pbcm} and so is the square in the front. 
	
	The outer square in \eqref{properhelp} therefore also is a pullback. But the outer square is independent of the choice of $i$ because $s \circ \pr_1 \circ d_i$ always is the map induced by $[0] \to [n],0 \mapsto 0$ since we excluded the case $i=0$. But then the left square in \eqref{properhelp} is also always a pullback.
\end{proof}
\begin{corollary} \label{overlineF}
	The functor \[\overline F \colon \Orbfld \longrightarrow \Top_{\Orb} \] from \Cref{lielifunctor} sends essential equivalences of orbifolds to weak equivalences of $\Orb$-spaces.
\end{corollary}
\begin{proof}
	Let $f \colon \G \to \HB$ be an essential equivalence of orbifolds. It also is an essential equivalence of topological groupoids, see \Cref{lievstop}. The induced map \[\gmap(\B K,\G) \to \gmap(\B K,\HB)\] is an essential equivalence for all universal subgroups $K$ of $\LI$ by \Cref{mapee}. By \Cref{fgree}, the induced map \[\| \gmap(\B K,\G)\|\to \|\gmap(\B K,\HB) \|\] is a weak equivalence. Consider the following commutative diagram \[\begin{tikzcd}
	\|\gmap(\B K,\G) \|\ar[r,"\simeq"] \ar[d,"\simeq"] & \| \gmap(\B K,\HB)\| \ar[d,"\simeq"] \\
	\left| \gmap(\B K,\G)\right| \ar[r] & \left|\gmap(\B K,\HB)\right|
	\end{tikzcd}\]
	where the vertical maps are the natural quotient maps and the horizontal maps are induced by $f$. The vertical maps are weak equivalences by \cite[Prop. A.1(iv)]{seg74} which applies because of \Cref{isgood}. We just checked that the upper map also is a weak equivalence. Therefore, the map \[f_\ast \colon |\gmap(\B K,\G)| \to |\gmap(\B K,\HB)|\] is a weak equivalence for all universal subgroups $K$ of $\LI$.
\end{proof}
\subsection{Orbifolds as \texorpdfstring{$\LI$}{L}-spaces} \label{comparison2}
We defined a functor \[\overline F \colon \Orbfld \longrightarrow \Top_{\Orb} \] in \Cref{lielifunctor}.
Moreover, there is a zig-zag of Quillen equivalences \[\begin{tikzcd}
\Top_{\Orb} \ar[r,"(f_1)^\ast",shift right,swap=0.6ex] & \Top_{\Orb^\prime} \ar[l,swap,shift right =0.6ex,"(f_1)_!"] \ar[r,shift left=0.6ex,"(f_2)_!"] & \Top_{\Ogl} \ar[l,shift left=0.6ex,"(f_2)^\ast"] \ar[r,shift left=0.6ex,"\Lambda"] & \LI\dash\Top \ar[l,shift left = 0.6ex,"\Phi"]  
\end{tikzcd} \] reviewed in \Cref{comparison} and \Cref{elmendorfLI}.
Using this we explain how to construct a functor \[\lieli \colon \Orbfld \longrightarrow \LI\dash\Top \] by post-composing $\overline F$ with the Quillen equivalences and a cofibrant replacement. Finally, we prove that the functor $\lieli$ coincides with the assignment $\effli$ from \Cref{functorLI} up to global equivalence of $\LI$-spaces.

Let $\Repl_{\cof}$ be a cofibrant replacement functor in $\Top_{\Orb^\prime}$. Such a functor always exists in the projective model structure on a small topological category, see \cite[Thm. 5.7]{pia91}. 
\begin{definition}[Functor $\lieli$] \label{lieli}
	Let $\lieli$ denote the composition \[\lieli=\Lambda \circ (f_2)_{!} \circ \Repl_{\cof} \circ (f_1)^\ast \circ \overline F\colon \Orbfld \longrightarrow \LI\dash\Top \,.\]
\end{definition}
\begin{corollary} \label{lielihomotopical}
	The functor $\lieli$ sends essential equivalences of orbifolds to global equivalences of $\LI$-spaces.
\end{corollary}
\begin{proof}
	The functor $\overline F \colon \Orbfld \to \Top_{\Orb}$ is homotopical by \Cref{overlineF}. The functor $(f_1)^\ast \colon \Top_{\Orb} \longrightarrow \Top_{\Orb^\prime}$ is homotopical because it is a right Quillen functor and all objects in $\Top_{\Orb}$ are fibrant. The functor $\Lambda \circ (f_2)_{!} \colon \Top_{\Orb^\prime} \longrightarrow \LI\dash\Top$ is left Quillen and hence homotopical on cofibrant objects.
\end{proof}
\begin{proposition} \label{generalHausdorff}
	Let $A$ be a Hausdorff $G$-space for $G$ a universal subgroup of $\LI$. Then there is a zig-zag of global equivalences between \[(\Lambda \circ (f_2)_{!} \circ \Repl_{\cof} \circ (f_1)^\ast \circ F)(G \ltimes A)\] and $\LI \times_G A$ in $\LI\dash\Top$.
\end{proposition}
In the case where $G \ltimes A$ admits the structure of a proper foliation Lie groupoid, this compares in particular the two constructions from \Cref{effli} and \Cref{grouporborb}, see also \Cref{lievstop} for a comparison between effective orbifolds and orbifolds.
\begin{corollary} \label{fororbifolds}
	Let $M$ be a manifold together with an almost free action of a universal subgroup $G$ of $\LI$ on $M$. There is zig-zag of global equivalences between $\lieli(G \ltimes M)$ and~$\LI \times_G M$. In particular, $\effli(\chi)=\LI \times_{\OG(n)} \Fr(\chi)$ and $\lieli(\OG(n) \ltimes \Fr(\chi))$ are globally equivalent for any effective orbifold $\chi$.
\end{corollary}
\begin{proof}[Proof of \Cref{generalHausdorff}]
	We will construct an object $X$ in $\Top_{\Orb^\prime}$ together with a weak equivalence from~$X$ to~$((f_1)^\ast \circ F)(G \ltimes A)$ and a weak equivalence from~$X$ to~$\left((f_2)^\ast \circ \Phi\right)(\LI \times_G A)$. For a moment, let us assume that such an object $X$ exists. 
	Pre-composing the second weak equivalence with the natural weak equivalence $\Repl_{\cof}X \to X$ yields a map \[\Repl_{\cof}X \longrightarrow \left((f_2)^\ast \circ \Phi\right)(\LI \times_G A)\] with adjoint \[(\Lambda \circ (f_2)_{!})(\Repl_{\cof}X) \longrightarrow \LI \times_G A\] which also is a weak equivalence because the Quillen pairs involved here are Quillen equivalences and all objects are fibrant in the universal projective model structure on~$\LI\dash\Top$. 
	Left Quillen functors are homotopical on cofibrant objects, so there is an induced weak equivalence \[(\Lambda \circ (f_2)_{!})(\Repl_{\cof}X) \longrightarrow (\Lambda \circ (f_2)_{!})\left( \left( \Repl_{\cof} \circ (f_1)^\ast \circ F \right)(G \ltimes A)\right) \,, \] which proves the claim.
\end{proof}
We will now construct an object $X$ in $\Top_{\Orb^\prime}$ satisfying the assumptions made above. 
If $A=\ast$, then $F(G \ltimes A)$ and $\Phi(\LI \times_G A)$ are the functors which represent the object $G$ (in $\Orb$ or $\Ogl$, respectively). It follows from the results in Körschgen's paper \cite{kor16} that we may choose $X$ to be the object in $\Top_{\Orb^\prime}$ which is represented by $G$. Therefore, the strategy for general $A$ is to adapt the constructions from \cite{kor16}. A brief overview of these constructions can be found in \Cref{comparison}.
We want to remind the reader that we are using left $G$-spaces while Körschgen uses right $G$-spaces.

For the rest of this section let us fix a universal subgroup $G$ of $\LI$ and a Hausdorff $G$-space $A$.
\begin{construction} \label{compareX}
	Let $K$ be a universal subgroup of $\LI$. The group $K \times G$ acts on $\LI$ via~$(k,g).\varphi=k\varphi g^{-1}$ for $k \in K$, $g \in G$ and $\varphi \in \LI$. We define \[\tilde \EC(K,G,A) = \setdef{(\alpha,\varphi,x) \in \map(K,G) \times \LI \times A}{\varphi \in \LI^\alpha,x \in A^{\alpha(K)}} \,. \] This is a $G$-space via \[g.(\alpha,\varphi,x) = (\alpha^g,\varphi g^{-1},g.x) \] for $g$ in $G$.  
	We define the $\Orb^\prime$-space $X$ by \[X(K) = \tilde \EC(K,G,A) \times_G \EC G \cong|G \ltimes\tilde \EC(K,G,A)| \] for a universal subgroup $K$ of $\LI$. Functoriality is defined similarly as is \cite[Def. 3.5]{kor16}: We define a composition law of topological groupoids \[(K \ltimes \tilde \EC(H,K)) \times (G \ltimes \tilde \EC(K,G,A)) \longrightarrow G \ltimes \tilde \EC(H,G,A)\] which is given by
	\begin{align*}
	K \times \tilde \EC(H,K) \times G \times \tilde \EC(K,G,A) & \longrightarrow G \times \tilde \EC(H,G,A) \,, \\
	(k,\alpha,\varphi) , (g,\beta,\psi,x) & \longmapsto (g\beta(k),\beta \circ \alpha,\varphi \psi,x)
	\end{align*}
	on arrows. The verification that this indeed gives rise to a functor between topological groupoids is analogous to the one in \cite[Prop. 3.7]{kor16}. Using that \begin{align*} {}&\left|K \ltimes \tilde\EC(H,K) \right| \times \left| G \ltimes \tilde\EC(K,G,A) \right| \\ \cong{}& \left|\left(K \ltimes  \tilde\EC(H,K) \right) \times \left(G \ltimes \tilde\EC(K,G,A) \right)\right| \end{align*} by \Cref{productfgr}, we obtain a map \[ \left| K \ltimes \tilde E(H,K) \right|\times \left|G \ltimes \tilde\EC(K,G,A)\right| \longrightarrow \left|G \ltimes \tilde\EC(H,G,A)\right| \] and hence a map \[\Orb^\prime(H,K) \times X(K) \longrightarrow X(H) \,. \] In the case where $A=\ast$, this coincides with the composition maps established by Körschgen in \cite[Prop. 3.9]{kor16}. The verification of the associativity and unitality with respect to the composition in $\Orb^\prime$ are a computation on the level of groupoids similar to the one by Körschgen for the case~$A=\ast$ in \cite[Prop. 3.8]{kor16}. Afterwards, one uses that the geometric realization is strongly monoidal, compare \cite[Prop. 3.9]{kor16}. We hence obtain an $\Orb^\prime$-space~$X$.
\end{construction}
We are now going to define the weak equivalence of $\Orb^\prime$-spaces \[l \colon X \longrightarrow ((f_1)^\ast \circ F)(G \ltimes A) \,.\]
\begin{construction}
	Let $K$ be a universal subgroup of $\LI$. Using that \[\gmap(\B K,G \ltimes A)_0 \cong \setdef{(\alpha,x) \in \map(K,G) \times A}{x \in A^{\alpha(K)}} \] with $G$-action given by $g.(\alpha,x) = (\alpha^g,g.x)$,  there is a $G$-equivariant (restricted) projection map \[\rho_A^K \colon \tilde \EC(K,G,A) \longrightarrow \gmap(\B K,\G \ltimes A)_0  \] such that the following diagram is a pullback: \[\begin{tikzcd}[row sep=large]
	\tilde \EC(K,G,A) \ar[r] \ar[d,swap,"\rho_A^K"] & \tilde \EC(K,G) \ar[d,"\rho^K"]  \\
	\gmap(\B K,G \ltimes A)_0 \ar[r] & \map(K,G)
	\end{tikzcd}\]
	where the horizontal maps are also projections (forgetting the factor $A$) and the map $\rho^K=\rho^K_\ast$ is also defined in \cite[Def. 2.12]{kor16}. 
	The map $\rho_A^K$ gives rise to the map \[\rho_A^K \times_G \EC G \colon \tilde E(K,G,A) \times_G \EC G \longrightarrow \gmap(\B K,G \ltimes A)_0 \times_G \EC G \,. \]
	Note that \begin{align*}\left((f_1)^\ast \circ F\right)(G \ltimes A)(K) ={}& |\gmap(\B K,G \ltimes A)|\\ \cong{}&|G \ltimes \gmap(\B K,G\ltimes A)_0|\\ \cong{}& \gmap(\B K,G \ltimes A)_0 \times_G \EC G \,. \end{align*}
	Using this homeomorphism, the map $\rho_A^K \times_G \EC G$ rewrites as a map \[l(K) \colon X(K) \longrightarrow \left((f_1)^\ast \circ F \right)(G \ltimes A)(K)\] and a computation on the level of groupoids similar to \cite[Prop. 3.8]{kor16} shows that these maps are compatible with the action maps.
	We obtain a map $l \colon X \to \left( (f_1)^\ast \circ F \right)(G \ltimes A)$.
\end{construction}
\begin{lemma} \label{firstcompare}
	The map \[l \colon X \longrightarrow \left((f_1)^\ast \circ F \right)(G \ltimes A)\] is a weak equivalence of $\Orb^\prime$-spaces. 
\end{lemma}
\begin{proof}
	Let $K$ be a universal subgroup of $\LI$.
	The map \[\rho^K \colon \tilde\EC(K,G) \longrightarrow \map(K,G)\] is a fiber bundle with contractible fiber by \cite[Prop. 2.15]{kor16} and so is the pullback \[\rho_A^K \colon \tilde \EC(K,G,A) \longrightarrow \gmap(\B K,\G \ltimes A)_0 \,.\] The rest of the proof is completely analogous to the one in \cite[Prop. 2.18]{kor16}. We basically use that the homotopy quotient construction is homotopical, so the induced map $\rho_A^K \times_G EG$ is again a weak equivalence and so is $l(K)$ by definition. 
\end{proof}
Now we are going to define the weak equivalence of $\Orb^\prime$-spaces \[\nu \colon X \longrightarrow \left((f_2)^\ast  \circ \Phi \right)(\LI \times_G A)\,.\]
\begin{lemma}
	Let $K$ be a universal subgroup of $\LI$. The map \begin{align*}\tilde E(K,G,A)/G & \longrightarrow (\LI \times_G A)^K \,, \\
	[\alpha,\varphi,x] & \longmapsto [\varphi,x]  \end{align*}is well-defined and a homeomorphism.
\end{lemma}
\begin{proof}
	We can adapt the proof of \cite[Prop. 2.20]{kor16}. More in detail, the proof does not use any properties of $\LI$ other than being a $(K \times G)$-space such that the restricted $G$-action is free. These assumption also hold for the space $\LI \times A$ where $K \times G$ acts via \[(k,g).(\varphi,x)=(k\varphi g^{-1},g.x)\] for $k \in K$, $g \in G$, $\varphi \in \LI$ and $x \in A$. We may replace $\LI$ with $\LI \times A$. The space $\tilde \EC(K,G)$ then becomes $\tilde \EC(K,G,A)$ and $(\LI/G)^K$ becomes $(\LI \times_G A)^K$.
\end{proof}
\begin{construction}
	Note that \[\left((f_2)^\ast \circ \Phi \right)(\LI \times_G A)(K) = \left(\LI\times_G A\right)^K \,.\] Using the above lemma, the canonical map \[\tilde \EC(K,G,A) \times_G \EC G \to \tilde \EC(K,G,A)/G\] defines a map \[\nu(K) \colon X(K) \longrightarrow \left((f_2)^\ast \circ \Phi \right)(\LI \times_G A)(K) \,. \] These maps assemble into a map $\nu \colon X \to \left((f_2)^\ast \circ \Phi \right)(\LI \times_G A)$. This can be proven similarly as in \cite[Prop. 3.12]{kor16} because functoriality on both sides does not interact with the elements of $A$.
\end{construction}
\begin{lemma} \label{secondcompare}
	The map \[\nu \colon X \longrightarrow \left((f_2)^\ast \circ \Phi \right)(\LI \times_G A)\] is a weak equivalence of $\Orb^\prime$-spaces.
\end{lemma}
\begin{proof} 
	The space $A$ is Hausdorff by assumption, the space $\map(K,G)$ is Hausdorff because it is metrizable, the space $\LI$ is Hausdorff by \cite[Prop. A.1]{sch19}. The product of Hausdorff spaces is Hausdorff and we conclude that the space $\tilde \EC(K,G,A)$ is Hausdorff because it is a subspace of a Hausdorff space. The natural map \[\tilde \EC(K,G,A) \times_G \EC G \longrightarrow \tilde \EC(K,G,A)/G\] from the homotopy quotient to the orbit space is a weak equivalence for all universal subgroups $K$ of $\LI$ by \cite[Thm. A.7]{kor16} because $\tilde \EC(K,G,A)$ is a free $G$-space and Hausdorff.
\end{proof}
\Cref{firstcompare} and \Cref{secondcompare} show that the object $X$ defined in \Cref{compareX} satisfies the assumptions made in the proof of \Cref{generalHausdorff}.

Finally, we want to verify that the underlying non-equivariant homotopy type of $\lieli(\G)$ is weakly equivalent to $|\G|$, compare \Cref{homotopygrouporbfld}. 
\begin{proposition} \label{nonequi}
	Let $\G$ be an orbifold. The underlying space of $\lieli(\G)$ is naturally weakly equivalent to $|\G|$.
\end{proposition}
\begin{proof}
	Let $\tg$ denote the trivial subgroup of $\LI$. The adjunction units $\Phi \circ \Lambda$ 
	and $(f_2)^\ast \circ (f_2)_!$ are equivalences on cofibrant objects because all objects in $\Top_{\Ogl}$, respectively $\LI\dash\Top$ with the universal projective model structure, are fibrant. We conclude that \begin{align*} \left(\lieli \G \right)^{\tg} ={}& (\Phi \circ \lieli)(\G)(\tg) \\ 
	={}& (\Phi \circ \Lambda \circ (f_2)_! \circ \Repl_{\cof} \circ (f_1)^\ast \circ \overline F)(\G)(\tg) \\ \simeq{}& ((f_2)_! \circ \Repl_{\cof} \circ (f_1)^\ast \circ \overline F)(\G)(\tg) \\
	={}& ((f_2)^\ast \circ (f_2)_! \circ \Repl_{\cof} \circ (f_1)^\ast \circ \overline F)(\G)(\tg) \\
	\simeq{}& \left(\Repl_{\cof} \circ (f_1)^\ast \circ \overline F\right)(\G)(\tg) \\
	\simeq{}& ((f_1)^\ast \circ \overline F)(\G)(\tg) \\
	={}& \overline F(\G)(\tg)=|\gmap(\B e,\G)|\cong{} |\G| \,. \end{align*} The weak equivalences are induced by adjunction units or by the cofibrant replacement functor and hence natural.
\end{proof}
\section{Towards Stable Global Homotopy Theory} \label{SGHB} 
\subsection{Orthogonal Spaces and Orthogonal Spectra} \label{ortspc}
We briefly recall some basic properties of orthogonal spaces and orthogonal spectra which are discussed in great detail in \cite{sch18}.
Orthogonal spaces are another model for unstable global homotopy types. But it is much easier to switch to orthogonal spectra which model stable global homotopy types.

We require all representations on inner product spaces to be orthogonal representations.
\begin{definition}[Orthogonal Spaces]
	Let $V$ and $W$ be two finite-dimensional inner product spaces. Let~$\LB(V,W)$ denote the space of linear isometric maps between $V$ and $W$ topologized as a Stiefel manifold, compare \cite[Sec. 1.1]{sch18}. 
	Let $\LB$ be the topological category with objects all finite-dimensional inner product spaces and morphism space $\LB(V,W)$ between two inner product spaces $V$ and $W$.
	An \defn{orthogonal space} is
	a continuous functor \[\LB \longrightarrow \Top \,.\] A \defn{morphism} between two orthogonal spaces is a continuous natural transformation between them. 
	We write $\Spc$ for the category of orthogonal spaces and morphisms between them. 
\end{definition}
\begin{definition}[Global equivalence]
	A morphism $f \colon X \to Y$ of orthogonal spaces is a \defn{global equivalence} if for every finite-dimensional $G$-representation $V$, every positive integer $n$ and every commutative diagram of the form \[\begin{tikzcd}[sep = large]
	S^{n-1} \ar[d,hook] \ar[r,"\alpha"] & X(V)^G \ar[d,"f(V)^G"] \\
	D^n \ar[r,,swap,"\beta"] & Y(V)^G \,,
	\end{tikzcd} \] there is a finite-dimensional inner product space $W$ together with a $G$-equivariant linear isometric embedding $\varphi \colon V \to W$ and a lift $l \colon D^n \to X(W)^G$ such that the in the following diagram  \[\begin{tikzcd}[sep = huge]
	S^{n-1} \ar[d,hook] \ar[r,"\alpha"] & X(V)^G \ar[r,"X(\varphi)^G"] & X(W)^G \ar[d,"f(W)^G"] \\
	D^n \ar[urr,dashed,"l"] \ar[r,swap,"\beta"] & Y(V)^G \ar[r,swap,"Y(\varphi)^G"] & Y(W)^G
	\end{tikzcd} \] the upper left triangle strictly commutes and the lower right triangle commutes up to homotopy relative $S^{n-1}$.
\end{definition}
\begin{remark}[Model structures on $\Spc$]
	The global equivalences are part of two model structures on the category of orthogonal spaces. 
	The \defn{global model structure} is established in \cite[Prop. 1.2.21]{sch18} and the \defn{positive global model structure} in \cite[Prop. 1.2.23]{sch18}.
\end{remark}
Orthogonal spaces are closely related to $\mathcal L$-spaces. 
\begin{construction} \label{atRinf}
	Let $X$ be an orthogonal space and let $\mathcal V$ be an inner product space of countably infinite dimension. Let $s(\mathcal V)$ denote the poset of finite-dimensional subspaces of $\mathcal V$, ordered by inclusion. We define the value of $X$ at $\mathcal V$ by \[X(\mathcal V) = \colim\limits_{V \in s(\mathcal V)} X(V) \,. \]
	One can also construct natural action maps \[\LB(\mathcal V,\mathcal W) \times X(\mathcal V) \longrightarrow X(\mathcal W) \] for another inner product space $\mathcal W$ of countably infinite dimension as explained in \cite[Con. 3.2]{sch19}. 
	In particular, $X(\R^\infty)$ inherits an action of $\LB(\R^\infty,\R^\infty)=\LI$.
\end{construction}
Now we can justify the definition of global equivalences between orthogonal spaces. 
\begin{proposition}[{\cite[Prop. 1.1.17]{sch18}}]
	Let $X$ and $Y$ be orthogonal spaces where all structure maps are closed embeddings. A morphism $f \colon X \to Y$ is a global equivalence of orthogonal spaces if and only if the induced map $f(\R^\infty) \colon X(\R^\infty) \to Y(\R^\infty)$ is a global equivalence of $\LI$-spaces.
\end{proposition}
Evaluation of an orthogonal space at $\R^\infty$ as described in \Cref{atRinf} is closely related to the left Quillen functor~$L$ which realizes the following Quillen equivalence.
\begin{theorem}[{\cite[Thm. 3.9]{sch19}}] \label{LSOS}
	There is a Quillen equivalence \[\LLO \colon \Spc \Leftrightarrows \LI\dash\Top \noloc \RLO \] between the global model structure on $\LI\dash\Top$ and the positive global model structure on~$\Spc$.
\end{theorem}
\begin{lemma}[{\cite[Prop. 3.7]{sch19}}] \label{WEbySch}
	There is a natural morphism of $\LI$-spaces \[\xi(X) \colon \LLO(X) \longrightarrow X(\R^\infty) \] which is a global equivalence if $X$ is cofibrant in the global model structure on orthogonal spaces.
\end{lemma}
Evaluation of an orthogonal space at a faithful $G$-representation admits the following left adjoint functor which relates the categories $G\dash\Top$ and $\Spc$.
\begin{definition}[Semifree orthogonal space]
	Let $V$ be a finite-dimensional faithful $G$-re\-pre\-sen\-ta\-tion. We define the \defn{semifree orthogonal space} functor \[\LB_{G,V}\colon G\dash\Top \longrightarrow \Spc \] as follows: Let $A$ be a $G$-space. For an inner product space $W$, we set \[(\LB_{G,V}A)(W)=\LB(V,W) \times_G A\] where $G$ acts on $\LB(V,W)$ be pre-composing the action on $V$. The structure maps act by post-composition. 
\end{definition}
Now we come to the definition of orthogonal spectra. 
\begin{construction}[Category $\OC$]
	Let $V$ and $W$ be finite-dimensional inner product spaces. We define \[\xi(V,W) = \defset{(w,\varphi) \in W \times \LB(V,W)}{w \perp \im(\varphi)}  \,. \] Together with the projection onto the second factor $\xi(V,W) \to \LB(V,W)$, this becomes a vector bundle. Let $\OC(V,W)$ be the Thom space of this bundle, that is the one-point compactification of the total space. Let $U$ be another inner product space. The map \begin{align*}
	\xi(V,W) \times \xi(U,V) & \longrightarrow \xi(U,W) \,,\\
	((w,\varphi),(v,\psi)) & \longmapsto (w+\varphi(v),\varphi\psi)
	\end{align*}
	extends to a map \[\circ \colon \OC(V,W) \times \OC(U,V) \longrightarrow \OC(U,W) \,. \]
	Hence we can define the $\Top_\ast$-enriched category $\OC$ with objects all finite-dimensional inner product spaces and morphism space $\OC(V,W)$ between two inner product spaces $V$ and~$W$. 
\end{construction}
\begin{definition}[Orthogonal $G$-spectrum] Let $G$ be a compact Lie group. 
	An \defn{orthogonal $G$-spectrum} is a based continuous functor \[ \OC \longrightarrow G\dash\Top_\ast \,. \]
	
	A \defn{morphism} of orthogonal $G$-spectra is a based continuous natural transformation of such functors.
	We write $G\dash\Sp$ for the category of orthogonal $G$-spectra. 
	
	Let $e$ denote the trivial group. An \defn{orthogonal spectrum} is an orthogonal $e$-spectrum. We write $\Sp$ for $\tg\dash\Sp$.
\end{definition}
\begin{remark} \label{GonXV}
	Let $X$ be an orthogonal $G$-spectrum and let $V$ be a representation of a compact Lie group $G$. The space $X(V)$ admits an action of the group $G \times G$ via the given action on the space itself and the action which is induced by the action on $V$. We usually consider~$X(V)$ as a $G$-space by the diagonal action.
	Note that the $G$-action on $X(V)$ may be non-trivial even if the orthogonal $G$-spectrum takes values in trivial $G$-spaces.
\end{remark}
\begin{definition}[Homotopy groups]
	Let $X$ be an orthogonal $G$-spectrum and let $H$ be a closed subgroup of $G$. Let $\mathcal U_H$ be a complete $H$-universe. Let $s(\mathcal U_H)$ be the poset of all finite-dimensional subrepresentations of $\mathcal U_H$. 
	Let $k$ be an non-negative integer. We define the \defn{$k$-th homotopy group} of $X$ as \[\pi_k^{H}(X) = \colim\limits_{V \in s(\mathcal U_H)} \left[S^{V \oplus \R^k},X(V)\right]^H_\ast \] and similarly the \defn{$(-k)$-th homotopy group} \[\pi_{-k}^{H}(X) = \colim\limits_{V \in s(\mathcal U_H)} \left[S^V,X(V \oplus \R^{k})\right]^H_\ast \,. \] 
	The maps in the colimit system are defined by using the suspension maps, see \cite[Sec, 3.1]{sch18} for a complete definition.
\end{definition}
\begin{definition}[$\underline \pi_\ast$-isomorphism]
	A morphism $f \colon X \to Y$ of orthogonal $G$-spectra is called a \defn{$\underline\pi_\ast$-isomorphism} if $\pi_k^H(f) \colon \pi_k^H(X) \to \pi_k^H(Y)$ is an isomorphism for all closed subgroups~$H$ of $G$ and all integers $k$. 
\end{definition}
\begin{definition}[Suspension spectrum] \label{sustopsp} We define the \defn{suspension} functor \[\Sigma_{G,+}^\infty \colon G\dash\Top \longrightarrow G\dash\Sp \] as follows: For $A$ a $G$-space and $V$ an inner product space, let $(\Sigma_{G,+}^\infty A)(V)=S^V \wedge A_+$ where the action of $G$ is given by the action on $A$ and the structure maps act on $S^{V}$. A $G$-equivariant map $f \colon A \to B$ induces maps $S^V \wedge A \to S^V \wedge B $ which commute with the structure maps.
	We write $\Sigma_+^\infty$ for $\Sigma_{e,+}^\infty$.
\end{definition}
\begin{definition}[Suspension spectrum] \label{suspensionspectrum}
	We define the \defn{suspension spectrum} functor \[\Sigma^\infty_+ \colon \Spc \longrightarrow \Sp \] with value at an orthogonal space $X$ and an inner product space $V$ given by \[(\Sigma^\infty_+ X)(V)=S^V \wedge X(V)_+ \,.\] 
	A linear isometry acts via the diagonal action on both factors while the suspension coordinate only affects the first factor. A full definition can be found in \cite[Con. 4.1.7]{sch18}. 
\end{definition}
This functor is left Quillen with respect to the global model structures, see \cite[Thm. 4.3.17(v)]{sch18}, and takes global equivalences of orthogonal spaces to global equivalences of orthogonal spectra, see \cite[Cor. 4.1.9]{sch18}. For an orthogonal space with constant value $A$, the functor $\Sigma_+^\infty$ coincides with the one from \Cref{sustopsp}.
\begin{definition}[Homotopy categories]
	Let $G\dash\SH$ denote the homotopy category of~$G\dash\Sp$ with respect to the $\underline\pi_\ast$-isomorphisms. We also simply write $\SH$ instead of~$\tg\dash\SH$.
	Let $\GH$ denote the homotopy category of $\Sp$ with respect to the global equivalences.
	
	These categories admit the structure of a triangulated category, the suspension functor is induced by the smash functor $(- \wedge S^1)$, compare \cite[Sec. 4.4]{sch18}.
\end{definition}
The category $G\dash\SH$ models $G$-equivariant cohomology theories. 
\begin{definition}
	Let $E$ be an orthogonal $G$-spectrum and $A$ be a $G$-space. We define the \defn{$k$-th $G$-equivariant cohomology group} represented by $E$ of $A$ as \[E^k(A)=G\dash\SH(\Sigma_{G,+}^\infty A,E[k]) \,. \]
\end{definition}
The trivial action functor \[(-)_G \colon \Sp \longrightarrow G\dash\Sp\] sends global equivalences to $\underline \pi_\ast$-isomorphisms 
and provides a connection between global homotopy theory which is studied in $\GH$ and $G$-equivariant homotopy theory which is studied in $G\dash\SH$.
\begin{proposition} \label{cohom}
	Let $A$ be a $G$-CW complex and let $E$ be an orthogonal spectrum. There is an isomorphism \[(E_G)^k(A) \cong \GH\left(\Sigma_+^\infty \LB_{G,V} A,E[k]\right) \] which is natural in $A$. 
\end{proposition}
\begin{proof}
	This is stated without a detailed proof in \cite{sch18}. It follows from the adjunctions described in \cite[Rem. 4.5.25]{sch18}: The trivial action functor admits a derived left adjoint and the value of this adjoint at $\Sigma_{G,+}^\infty A$ is $\Sigma_+^\infty \LB_{G,V}A$.  
\end{proof}
\subsection{The Homotopy Type of an Orbifold}\label{htorb}
In this section we define orbifold cohomology groups and discuss various properties and examples. 

The functor $\lieli$ from \Cref{lieli} associates an $\LI$-space to an orbifold. By first using a fibrant replacement in the global model structure on $\LI\dash\Top$ and then the right Quillen functor \[\RLO \colon \LI\dash\Top \longrightarrow \Spc\] from \Cref{WEbySch} we further obtain an orthogonal space associated to an orbifold. Afterwards, we use the homotopical left Quillen functor \[\Sigma_+^\infty \colon \Spc \longrightarrow \Sp\] from \Cref{suspensionspectrum} and finally obtain an orthogonal spectrum which is supposed to model the stable global homotopy type of the orbifold.
Let \[\Repl_{\fib} \colon \LI\dash\Top \to \LI\dash\Top\] be a fibrant replacement functor in $\LI\dash\Top$. The existence of a functorial replacement is proven when establishing the global model structure on $\LI\dash\Top$, see \cite[Thm. 1.20]{sch19}.
\begin{definition}[Orbifold cohomology]
	Let $\G$ be an orbifold and let $E$ be an orthogonal spectrum and let $k$ be an integer. We define the \defn{$k$-th orbifold cohomology group} represented by $E$ of $\G$ by \[E^k(\G)= \GH\left(\left(\Sigma_+^\infty\circ \RLO\circ\Repl_{\fib}\circ\lieli\right)(\G), E[k] \right) \,.\] This extends to a functor $E^k({-}) \colon \Orbfld^{\op} \to \Ab$ and morphisms between orthogonal spectra induce natural transformations between the associated orbifold cohomology theories. 
\end{definition}
\begin{proposition}
	Let $E$ be an orthogonal spectrum and let $k$ be an integer. The functor $E^k({-}) \colon \Orbfld^{\op} \to \Ab$ inverts essential equivalences of orbifolds. 
\end{proposition}
\begin{proof}
	The functor $\lieli$ is homotopical by \Cref{lielihomotopical} and so are $\RLO \circ \Repl_{\fib}$ and $\Sigma_+^\infty$.
\end{proof}
Let $\Orb^{\fin}$ denote the full subcategory of finite groups of $\Orb$. The category $\Top_{\Orb^{\fin}}$ models unstable global homotopy types for the family of finite groups, see \cite[Sec. 1.3]{gh07}. The forgetful functor $\Top_{\Orb} \to \Top_{\Orb^{\fin}}$ is right Quillen with adjoint $\Top_{\Orb^{\fin}} \to \Top_{\Orb}$ given by left Kan extension. Every morphism~$\B K \to \G$ from a universal subgroup $K$ of $\LI$ to an orbifold $\G$ factors through a stabilizer of $\G$ and hence through a finite group. The $\Orb$-space $\overline F(\G)$ hence is contained in the essential image of the left adjoint. Similarly, there are analogs of $\Top_{\Orb^{\prime}}$ and $\Top_{\Ogl}$ for the family of finite groups. There also are $\fin$-global model structures on $\LI\dash\Top$ and $\Spc$ (\cite[Rem. 3.11]{sch19}) such that the identity is a left Quillen functor from the $\fin$-global model structure to the global model structure. All established Quillen equivalenes expand to the $\fin$-relative versions and these Quillen equivalences commute with the forgetful functors from the universal to the $\fin$-relative versions. We conclude that the orthogonal space $\left(\RLO\circ \Repl_{\fib} \circ \lieli \right)(\G)$ is also left induced from the $\fin$-global model structure on $\Spc$, up to global equivalence. 
There is also a $\fin$-global model structure on $\Sp$ and the identity is a left Quillen functor from this $\fin$-global model structure to the global model structure on $\Sp$. 
The suspension spectrum functor $\Sigma_+^\infty\colon \Spc \to \Sp$ is left Quillen with respect to the respective $\fin$-global model structures, see \cite[Thm. 4.3.17(v)]{sch18}. The global homotopy type $(\Sigma_+^\infty \circ \RLO \circ \Repl_{\fib} \circ \lieli) (\G)$ hence is left induced from the family of finite groups in the sense of \cite[Thm. 4.5.1(i)]{sch18}. This in particular proves the following:
\begin{proposition} \label{fin}
	A $\fin$-global equivalence between two orthogonal spectra induces isomorphisms on orbifold cohomology groups for all orbifolds. 
\end{proposition}
We are now going to verify that the orbifold cohomology of a global quotient orbifold can be identified with the equivariant cohomology of the underlying manifold. 
\begin{lemma} \label{orbfldspc}
	Let $G$ be a universal subgroup of $\LI$, $V$ be a faithful non-zero $G$-representation and $A$ a $G$-CW complex. There is a zig-zag of global equivalences \[\left(\RLO \circ \Repl_{\fib} \right)(\LI \times_G A) \simeq \LB_{G,V}A \,. \]
\end{lemma}
\begin{proof}
	The orthogonal space $\LB_{G,V}A$ is cofibrant in the global model structure. Furthermore, $\LB_{G,V}A(0)=\emptyset$ and therefore $\LB_{G,V}$ is also cofibrant in the positive global model structure, see \cite[Def. 1.2.22]{sch18}. The composition \[\LLO ( \LB_{G,V}A ) \xrightarrow{\xi(\LB_{G,V}A)} (\LB_{G,V}A)(\R^\infty) \cong \LB(V,\R^\infty) \times_G A \simeq \Repl_{\fib}(\LB(V,\R^\infty) \times_G A) \] is a global equivalence by \Cref{WEbySch}. By \Cref{LSOS}, the adjoint map \[\LB_{G,V}A \longrightarrow \RLO( \Repl_{\fib}(\LB(V,\R^\infty) \times_G A))\] also is a global equivalence. Moreover, \[\LB(V,\R^\infty) \times_G A \simeq \LI \times_G A \] by \cite[Prop. 1.10]{sch19}. The claim follows because the functor $\RLO$ is right Quillen and therefore preserves global equivalences between fibrant objects. 
\end{proof}
\begin{corollary} \label{cohomorbi}
	Let $G$ be a compact Lie group acting almost freely on a manifold $M$. The orbifold cohomology of $G \ltimes M$ is isomorphic to the $G$-equivariant cohomology of $M$, i.e. \[E^k(G \ltimes M) \cong E_G^k(M) \] for any orthogonal spectrum $E$.
\end{corollary}
\begin{proof}
	By \Cref{orbfldspc} and \Cref{fororbifolds}, we have \[\left(\RLO\circ\Repl_{\fib} \circ\lieli\right)(G \ltimes M) \simeq \LB_{G,V}M\] for an almost free action of a compact Lie group $G$ on a manifold $M$. The rest is a reformulation of \Cref{cohom}. 
\end{proof}
For $\tilde U$ an open subset of the object manifold $\G_0$ of a Lie groupoid $\G$, we define the \defn{restricted groupoid $\res{\G}{\tilde U}$} of $\G$ to $\tilde U$ as the Lie groupoid with object manifold $\tilde U$ and arrow manifold $(s,t)^{-1}(\tilde U)$. The structure maps are given by restricting the corresponding maps of $\G$. We moreover write \[\G_0/\G_1=\coeq(s,t \colon \G_1 \longrightrightarrows \G_0)\] for the \defn{orbit space} of $\G$.
\begin{proposition}[Mayer-Vietoris sequence] \label{MVS}
	Let $\G$ be an orbifold and let $U,V$ be two open subsets covering the orbit space $\G_0/\G_1$. Let $\tilde U$ and $\tilde V$ be the preimages of $U$ and $V$ under the quotient map $\G_0 \to \G_0/\G_1$. There is a long exact sequence of orbifold cohomology groups \[\begin{tikzcd}[sep=2.7ex]
	\dots \ar[r] & E^k\left(\G\right) \ar[r,"{\left(i_U^\ast,i_V^\ast\right)}"] &[3.8ex] E^k\left(\res{\G}{\tilde U} \right) \oplus E^k\left(\res{\G}{\tilde V} \right) \ar[r,"{j_U^\ast - j_V^\ast}"] &[3.8ex] E^k\left(\res{\G}{\tilde U \cap \tilde V} \right) \ar[r] & E^{k+1}(\G) \ar[r] & \dots 
	\end{tikzcd} \] 
	for every orthogonal spectrum $E$. Here, $i_U \colon \res{\G}{\tilde U} \to \G$, $i_V \colon \res{\G}{\tilde V} \to \G$, $j_U \colon \res{\G}{\tilde U \cap \tilde V} \to \res{G}{\tilde U}$ and $j_V \colon \res{\G}{\tilde U \cap \tilde V} \to \res{\G}{\tilde V}$ denote the inclusion morphisms. 
\end{proposition} 
\begin{proof}
	We first check that the commutative diagram of orbifolds \begin{equation} \label{mvpushout} \begin{tikzcd}[sep=large]
	\res{\G}{\tilde U \cap \tilde V} \ar[r,"j_U"] \ar[d,swap,"j_V"] & \res{\G}{\tilde U} \ar[d,"i_U"] \\
	\res{\G}{\tilde V} \ar[r,"i_V"] & \G
	\end{tikzcd} \end{equation} is sent to a pushout square in $\Top$ under the functor $|\gmap(\B K, -)| \colon \Orbfld \to \Top$ for a universal subgroup $K$ of $\LI$. The geometric realization is a left adjoint and therefore commutes with pushouts. We can hence also check that the functor which assigns the $n$-th space of the nerve of $\gmap(\B K,\G)$ to a topological groupoid $\G$ sends \eqref{mvpushout} to a pushout. It can be checked on elements that this is true for the underlying sets. It is also a pushout in $\Top$ because all maps are sent to inclusions of open subsets.
	
	The natural map from the homotopy pushout to a pushout along open inclusions is a weak equivalence, see \cite[Cor. 1.6]{di04}. This applies to the above situation. The functor $\overline F$ from \Cref{lielifunctor} hence takes \eqref{mvpushout} to a levelwise homotopy pushout in $\Top_{\Orb}$. 
	
	We are now going to prove that a levelwise homotopy pushout in $\Top_{\Orb}$ is a homotopy pushout. Let $P$ denote the diagram $(\begin{tikzcd}
	\cdot & \cdot \ar[r] \ar[l] & \cdot 
	\end{tikzcd})$. 
	Pushouts can be computed levelwise. The homotopy pushout functor in $\Top_{\Orb}$ therefore is the left derived functor of \[\left(\colim_P\right)^{\Orb} \colon [P,\Top_{\Orb}] \cong [\Orb,[P,\Top]] \longrightarrow [\Orb,\Top] = \Top_{\Orb} \,.\] 
	It follows from general model category theory that the homotopy pushout can be computed as the pushout of a cofibrant replacement in the projective model structure on $[P,\Top_{\Orb}]$, compare \cite[Cor. 5.1.6]{hov99}. The homotopy pushout in $\Top_{\Orb}$ hence is weakly equivalent to the levelwise homotopy pushout if a cofibrant diagram in $[P,\Top_{\Orb}]\cong[\Orb,[P,\Top]]$ levelwise consists of diagrams in $[P,\Top]$ for which the natural map from the homotopy pushout to the pushout is a weak equivalence.
	This is the case for diagrams in $[P,\Top]$ where both arrows of $P$ are mapped to Hurewicz cofibrations, see e.g. \cite[Thm. A.7]{di04}. As can be seen from the characterization given in \cite[Thm. 5.1.3]{hov99}, a cofibrant object in $[P,\Top_{\Orb}]$ maps the two arrows of $P$ to cofibrations between $\Orb$-spaces and these cofibrations are levelwise Hurewicz cofibrations by \cite[Prop. 2.2(iv)]{kor17}.
	
	The rest of the proof is now a formal consequence. The functors $(f_1)_\ast \colon \Top_{\Orb} \to \Top_{\Orb^\prime}$, $\Lambda \circ (f_2)_! \circ \Repl_{\cof} \colon \Top_{\Orb^\prime} \to \LI\dash\Top$ and $\RLO \circ \Repl_{\fib} \colon \LI\dash\Top \to \Spc$ are lifts of derived Quillen equivalences and $\Sigma_+^\infty \colon \Spc \to \Sp$ is a homotopical left Quillen functor. All these functors hence preserve homotopy pushouts. We conclude that \eqref{mvpushout} is sent to a homotopy pushout in $\Sp$. The category $\Sp$ together with the global model structure is a stable model category and the homotopy pushout gives rise to a distingushed triangle in $\GH$ as explained in \cite[Lem. 5.7]{may01}. The claim follows by applying $\GH(-,E[k])$ to the distingushed triangle.
\end{proof}
\begin{proposition}[Additivity]
	Let $E$ be an orthogonal spectrum. Let $(\G^i)_{i \in I}$ be a collection of orbifolds indexed by a countable set $I$. The inclusions $\G^j \to \coprod_{i \in I} \G^i$ induce an isomorphism of orbifold cohomology groups \[E^k \left( \coprod_{i \in I} \G^i \right) \cong \prod_{i \in I} E^k(\G^i)  \,. \]
\end{proposition}
\begin{proof}
	The functor $\overline F$ from \Cref{functorLI} preserves coproducts because they can be computed levelwise and $|\gmap(\B K,-)| \colon \Orbfld \to \Top$ preserves coproducts for any universal subgroup $K$ of $\LI$ by a similar reasoning as for pushouts in \Cref{MVS}. It follows from general model category theory that a coproduct in the homotopy category of $\Top_{\Orb}$ can be computed as the coproduct of a cofibrant replacement, see \cite[Ex. 1.3.11]{hov99}. But coproducts preserve weak equivalences in $\Top_{\Orb}$ because they can be computed levelwise in $\Top$. A coproduct in $\Top_{\Orb}$ therefore also computes a coproduct in the homotopy category. 
	The following derived Quillen equivalences $(f_1)_\ast \colon \Top_{\Orb} \to \Top_{\Orb^\prime}$, $\Lambda \circ (f_2)_! \circ \Repl_{\cof} \colon \Top_{\Orb^\prime} \to \LI\dash\Top$ and $\RLO \circ \Repl_{\fib} \colon \LI\dash\Top \to \Spc$ are lifts of equivalences of homotopy categories. The homotopical left Quillen functor $\Sigma_+^\infty \colon \Spc \to \Sp$ is a lift of a left adjoint between the respective homotopy categories. All these functors hence preserve coproducts in the homotopy categories. 
	All in all, the orbifold $\coprod_{i \in I} \G^i$ is sent to a coproduct of $(\G^i)_{i \in I}$ in $\GH$. The claim follows by applying $\GH(-,E[k])$.
\end{proof}
We check that the underlying non-equivariant homotopy type of an orbifold is modeled by its classifying space. Let $\UF \colon \GH \to \SH$ denote the forgetful functor induced by the identity on $\Sp$.
\begin{proposition}\label{nonequisp} Let $\G$ be an orbifold. There is a natural isomorphism
	\[\UF \left(\left(\Sigma_+^\infty \circ \RLO \circ \Repl_{\fib}\circ \lieli \right)(\G)\right) \cong \Sigma_+^\infty |\G|  \] in the stable homotopy category $\SH$. 
\end{proposition}
\begin{proof}
	The functor $\Sigma_+^\infty$ is homotopical, so we can replace $\left(\RLO \circ\Repl_{\fib}\circ \lieli \right)(\G)$ cofibrantly. We may afterwards replace this orthogonal space by the underlying non-equivariant space of its evaluation at $\R^\infty$, see \cite[Prop. 5.14]{sch14}. But evaluation at $\R^\infty$ is derived equivalent to $L$ by \Cref{WEbySch} and the adjunction unit $\LLO \circ \RLO$ is derived equivalent to the identity. We obtain \begin{align*}
	\UF \left(\left(\Sigma_+^\infty \circ \RLO \circ \Repl_{\fib}\circ \lieli \right)(\G)\right) \cong{}& \Sigma_+^\infty \left(\left(\LLO \circ \Repl_{\cof} \circ \RLO \circ \Repl_{\fib}\circ \lieli \right)(\G)\right)^\tg \\
	\cong{}& \Sigma_+^\infty \left( \lieli  (\G) \right)^\tg  \cong \Sigma_+^\infty |\G|
	\end{align*} where the last isomorphism is established in \Cref{nonequi}.
\end{proof}
\begin{example}[Cohomology of the classifying space, Borel cohomology] 
	Let $E$ be a non-equivariant generalized cohomology theory. By Brown's representabiliy theorem, there is an orthogonal spectrum representing $E$ which we also denote by $E$.
	The forgetful functor $\UF \colon \GH\to\SH$ admits a right adjoint $b \colon \SH\to \GH$, see \cite[Thm. 4.5.1]{sch18}. Let $\G$ be an orbifold. By \Cref{nonequisp}, we obtain a chain of natural isomorphisms \begin{align*} (bE)^k(\G) ={}& \GH\left(\left( \Sigma_+^\infty\circ \RLO\circ\Repl_{\fib}\circ \lieli\right)(\G) ,bE[k] \right) \\ \cong{}& \SH \left( \UF \left( \left( \Sigma_+^\infty\circ \RLO\circ\Repl_{\fib}\circ \lieli\right)(\G) \right) , E[k] \right) \\ \cong{}& \SH\left(\Sigma_+^\infty |\G|, E[k]\right) = E^k(|\G|) \,. \end{align*}
	We did moreover use that the functor $b$ commutes with the respective shift functors in $\SH$ and $\GH$ because its adjoint $\UF$ does so. 
	We just showed that the orbifold cohomology theory represented by $bE$ can be identified with the cohomology represented by $E$ of the classifying space. For $G \ltimes M$ a global quotient orbifold, we have $|G \ltimes M| \cong M \times_G \EC G$. Therefore, $bE$ represents the equivariant Borel cohomology represented by $E$ for all global quotients orbifolds. 
	
	The functor $b \colon \SH \to \GH$ admits an explicit lift $\Sp \to \Sp$ which is described in \cite[Con. 4.5.21]{sch18}. 
\end{example}
\begin{example}[Bredon cohomology] \label{bredon} We briefly recall the definition of Bredon cohomology and refer the reader to \cite[Sec. 1.4]{blu17} for more details on Bredon cohomology. 
	Let $\Or{G}^{\fin}$ denote the orbit category with finite stabilizers of a compact Lie group $G$. Given a functor $F \colon \left(\Or{G}^{\fin}\right)^{\op} \to \Ab$ and a $G$-CW complex $A$, we define $C(A) \colon \left(\Or{G}^{\fin}\right)^{\op} \to \Cell(\Ab)$ by \[C_\ast(A)(G/H) = C^{\operatorname{cell}}_\ast (A^H) \] where $C^{\operatorname{cell}}_\ast(-)$ denotes the cellular chain complex. The \defn{Bredon cohomology} $H_G^\ast(A,F)$ of $A$ with coefficient system $F$ is defined as the cohomology of the chain complex $\Hom_{\Fun\left(\left(\Or{G}^{\fin}\right)^{\op},\Ab\right)}(C_\ast(A),F)$.
	
	Let $F$ be a global Mackey functor as defined in \cite[Sec. 4.2]{sch18}. By \cite[Thm. 4.4.9]{sch18}, there is an \defn{Eilenberg-MacLane spectrum} $HF$, i.e a spectrum with homotopy groups concentrated in degree~$0$, such that  $\pi^G_0(HF)\cong F(G)$ for all compact Lie groups $G$ with corresponding transfer and restriction maps. The underlying orthogonal $G$-spectrum $(HF)_G$ represents Bredon cohomology with coefficient system $F_G \colon \left(\Or{G}^{\fin}\right)^{\op} \to \Ab$, the underlying contravariant functor of the restriction of $F$ to a $G$-Mackey functor, see \cite[Ex. 3.2.16]{dhlps}.
	For $G \ltimes M$ a global qoutient orbifold, we obtain \[ HF^k(G \ltimes M) \cong HF_G^k(M) \cong H_G^k(M,F_G) \] by \Cref{cohomorbi}. We conclude that the spectrum $HF$ represents Bredon cohomology for global quotient orbifolds. This in particular proves that the Bredon cohomology of an orbifold is independent of the choice of the presentation as a global quotient orbifold. Adem and Ruan \cite[Sec. 3, Rem. 5.11]{ar03} proved this for the special case where $F=\RU_\Q$ is the rationalized representation functor, Pronk and Scull \cite[Prop. 5.11]{ps10} proved this for what they call an \emph{orbifold coefficient system}.
	
	Moreover, Bredon cohomology extends to an actual cohomology theory on all orbifolds. 
\end{example}
\begin{construction}[Atiyah-Hirzebruch spectral sequence]
	The full subcategories of globally connective spectra and globally coconnective spectra form a $t$-structure on $\GH$, see \cite[Thm. 4.4.9]{sch18}. The heart of this $t$-structure consists of those spectra whose homotopy groups are concentrated in degree~$0$, i.e. Eilenberg-MacLane spectra. The functor $\underline \pi_0$ induces an equivalence between this heart and the category of global Mackey functors. 
	Note that this $t$-structure is induced by a $t$-model structure on $\Sp$ in the sense of Fausk and Isaksen \cite[Def. 4.1]{fi07}. They give a short argument why this is true for model categories of spectra which also applies in our case.
	
	Their paper \cite[Sec. 10]{fi07} explains how to construct an Atiyah-Hirzebruch spectral sequence from a $t$-model structure on a stable model category. We briefly recall the most important ideas and adjust them to our situation:
	There is an $(q-1)$-connective cover $Y_{\ge q}$ for any spectrum $Y$ and any integer $q$ and natural transformations $Y_{\ge q+1} \to Y_{\ge q}$ in $\GH$. The homotopy cofiber of this map is a shifted Eilenberg-MacLane spectrum of type $\underline \pi_q(Y)$. 
	
	For $X$ another orthogonal spectrum, we apply $\GH(X,-[-p-q])$ to the distingushed triangle \[ \begin{tikzcd} Y_{\ge q+1} \ar[r] & Y_{\ge q} \ar[r] & \left(H \underline \pi_q(Y)\right)[q] \ar[r] & Y_{\ge q+1}[1] \end{tikzcd} \] to get an exact couple with $D^2_{p,q}=\GH(X,Y_{\ge q}[-p-q])$ and $E^2_{p,q}=\GH(X,\left(H \underline \pi_{q}(Y)\right)[-p])$.
\end{construction}
For $X=\left(\Sigma_+^\infty \circ R \circ \Repl_{\fib} \circ \lieli\right)(\G)$, we obtain the following:
\begin{proposition}
	Let $Y$ be an orthogonal spectrum and let $\G$ be an orbifold. There is a spectral sequence with $E^2_{p,q}=(H\underline \pi_q(Y))^{-p}(\G)$ -- Bredon cohomology of $\G$ with coefficient system $\underline \pi_q(Y)$ -- which conditionally converges to $Y^{p+q}(\G)$. 
\end{proposition}
Here, conditional convergences of spectral sequences is defined in Boardman's paper \cite[Def. 5.10]{boa99}. It also provides several criteria for deducing strong convergence from conditional convergence. 
\begin{proof}
	The statement follows from Fausk and Isaksen's theorem \cite[Thm. 10.1]{fi07}. We need to check that $\left(\Sigma_+^\infty \circ R \circ \Repl_{\fib} \circ \lieli\right)(\G)$ is bounded below and that $\holim_{n \to \infty} Y_{\ge n}$ is contractible. 
	
	The first condition holds true because suspension spectra are globally connective by \cite[Prop. 4.1.11]{sch18}. 
	
	Let $G$ be a compact Lie group and $k$ be an integer. The functor $\pi_0^G \colon \GH \to \Ab$ is corepresentable by \cite[Thm. 4.4.3(i)]{sch18}. We hence obtain a Milnor short exact sequence, see e.g. \cite[\href{https://stacks.math.columbia.edu/tag/0919}{Lem. 0919}]{stacks-project}: \[\begin{tikzcd}[sep=small]
	0 \ar[r] & \dlim_{n \to \infty} \pi^G_{k+1}(Y_{\ge n}) \ar[r] & \pi^G_k\left( \holim_{n \to \infty} Y_{\ge n} \right) \ar[r] & \lim_{n \to \infty} \pi_k^G(Y_{\ge n}) \ar[r] & 0
	\end{tikzcd} \] 
	The homotopy groups $\pi^G_{k}(Y_{\ge n})$ eventually vanish. The left and the right term of the sequence therefore vanish and so does the middle term. 
\end{proof}
\begin{example}[Orbifold $K$-theory] \label{ktheory}
	Firstly, we briefly discuss the geometric definition of orbifold $K$-theory via orbifold vector bundles. A full definition and proofs of the following statements can be found in \cite[Sec. 3.3]{alr07}. Let $\G$ be a compact orbifold, that is a proper foliation groupoid where $\G_0$ is compact. A \defn{$\G$-space} is a smooth manifold $M$ together with a smooth \defn{anchor map} $\alpha \colon M \to \G_0$ and a smooth \defn{action map}~$\mu \colon \G_1 \times_{\G_0} M \to M$ which satisfy the usual unitality and associativity conditions. 
	A $\G$-vector bundle is a $\G$-space $M$ where the anchor map $\alpha \colon M \to \G_0$ is a complex vector bundle such that the action of $\G$ is fiberwise linear. Similarly as for manifolds, one can define Whitney sums and tensor products of $\G$-vector bundles.
	
	The \defn{orbifold $K$-theory} $\KT_{\Orbfld}(\G)$ of $\G$ is defined as the Grothendieck ring of isomorphism classes of $\G$-vector bundles. This is a functorial construction (using pullbacks of vector bundles) and one can prove that every essential equivalence between orbifolds induces an isomorphism on these rings, compare \cite[Prop. 4.2]{ar03}. 
	
	Moreover, the $K$-theory of a global quotient orbifold $G \ltimes M$ may be identified with the $G$-equivariant $K$-theory of $M$, i.e. there is an isomorphism \[\KT_{\Orbfld}(G \ltimes M) \cong \KT_G(M) \,, \] which is natural in $M$, see \cite[Prop. 3.6]{alr07}. 
	
	On the other hand, there is the \defn{periodic global $K$-theory spectrum} $\KU$ which represents $K$-theory in the following sense: For every finite $G$-CW complex $A$, there is a natural isomorphism \[\KT_G(A) \cong \KU_G^0(A) \,. \] This applies when $A=M$ is a compact manifold together with a $G$-action on $M$. The definition of $\KU$ is due to Joachim \cite{joa04}. A self-contained exposition from a global perspective can be found in \cite[Sec. 6.4]{sch18}. 
	
	By also using \Cref{cohom}, we finally obtain an isomorphism \[\KT_{\Orbfld}(G \ltimes M) \cong \KU^0(G \ltimes M) \,. \] We conclude that the ring spectrum $\KU$ represents orbifold $K$-theory for compact global quotient orbifolds. Moreover, both sides are invariant under essential equivalences of orbifolds. We conclude hat $\KU$ represents orbifold cohomology for all compact orbifolds by \Cref{gqc}.
	
	We want to point out that it is however not clear that this isomorphism is natural with respect to all maps of orbifolds. 
\end{example}
\begin{example}[Rationalized $K$-theory]
	The rationalized $K$-theory is represented by the rationalized periodic global $K$-theory spectrum $\KU_\Q$. By \cite[Thm. 1.6]{wim18}, this spectrum decomposes up to $\fin$-global equivalence into a sum of shifted Eilenberg-MacLane spectra \[\KU_\Q \simeq \bigvee_{n \in \Z} H(\RU_\Q)[2n] \] which is enough for our purposes, see \Cref{fin}. This implies that there is an isomorphism \[\KU^0(\G) \otimes \Q \cong \bigoplus_{n \in \Z} (H(\RU_\Q))^{2n}(\G) \] between the rationalized $K$-theory and Bredon cohomology with coefficients system $\RU_\Q$ of an orbifold $\G$. For $\G=G \ltimes M$ a global quotient orbifold of a compact manifold $M$ under the action of a compact Lie group $G$, this rewrites as \[\KT_{\Orbfld}(G \ltimes M) \otimes \Q \cong \bigoplus_{n \in \Z} H_G^{2n}(M,\RU_\Q) \] by \Cref{ktheory} and \Cref{bredon}. This has been noted before by Adem and Ruan \cite[Rem. 5.11]{ar03}.
\end{example}
\printbibliography[heading=bibintoc]
\end{document}